\documentclass[10pt,reqno]{amsart}
\usepackage[english]{babel}
\usepackage{amsmath}
\numberwithin{equation}{section}
\usepackage{amsfonts}
\usepackage{amssymb}
\usepackage{mathrsfs}
\usepackage{amsthm}
\usepackage{appendix}
\usepackage[all]{xy}
\usepackage{graphicx}
\usepackage[usenames]{color}
\usepackage[margin=1in]{geometry}

\newcommand{\J}{{{\mathcal{J}_\varepsilon}}}
\newcommand{\R}{{{\mathbb R}}}

\newcommand{\N}{{{\mathbb N}}}
\newcommand{\Z}{{{\mathbb Z}}}
\newcommand{\T}{{{\mathbb T}}}

\newcommand{\C}{{{\mathbb C}}}

\providecommand{\dr}[1]{\partial_{r}#1}
\providecommand{\dx}[1]{\partial_{\xi}#1}
\providecommand{\dxx}[1]{\partial_{\xi}^{2}#1}
\providecommand{\de}[1]{\partial_{\theta}^{#1}}
\providecommand{\dfi}[1]{\partial_{\varphi}^{#1}}
\providecommand{\abs}[1]{\lvert#1\rvert}
\providecommand{\e}[1]{^{#1}}
\providecommand{\dhil}[1]{\dot{H}^{#1}(\T)}

\providecommand{\norm}[1]{\lVert#1\rVert}

\theoremstyle{plain}
\newtheorem{lem}{Lemma}[section]
\newtheorem{prop}{Proposition}[section]
\newtheorem{prop*}{Proposition}
\newtheorem{thm}{Theorem}[section]
\newtheorem{nota}{Remark}[section]

\theoremstyle{definition}
\newtheorem{hip}{Hypothesis}[section]

\newtheorem{innercustomgeneric}{\customgenericname}
\providecommand{\customgenericname}{}
\newcommand{\newcustomtheorem}[2]{%
  \newenvironment{#1}[1]
  {%
   \renewcommand\customgenericname{#2}%
   \renewcommand\theinnercustomgeneric{##1}%
   \innercustomgeneric
  }
  {\endinnercustomgeneric}
}

\newcustomtheorem{customprop}{Proposition}

\title[Analysis of a thin film approximation]{Analysis of a thin film approximation for two-fluid Taylor-Couette flows}
\author{Tania Pernas-Casta\~{n}o and Juan J. L. Vel\'azquez}

\begin{document}
\maketitle

\begin{abstract}
In this work we study the evolution of the interface between two different fluids in two concentric cylinders when the velocity is given by the Navier-Stokes equation and one of the fluids is thin. We present a formal asymptotic derivation of the evolution equation for the interface under different scaling assumptions for the surface tension. We then study the different types of the stationary solutions and travelling waves for the resulting equation. In particular, we state a global well posedness result and using  Center Manifold Theory, we obtain detailed information about the long time asymptotics of the solutions of the problem.
\end{abstract}
\section{introduction}

The fluid flow which arises when a viscous fluid is confined between two concentric cylinders is termed as the Taylor-Couette flow. This has been extensively studied in the physical and mathematical literature (cf. \cite{Baumert}, \cite{Chandra}, \cite{Chossat}, \cite{Drazin}, \cite{Schlichting}). Couette was the first who experimentally found fluid flows for which the streamlines of the fluid are circles concentric with the two cylinders. In 1923, G. I. Taylor studied mathematically the solution of the Navier-Stokes equations which describes this flow and he discovered that this solution becomes unstable if the difference of angular velocities of the two cylinder confining the liquid is sufficiently large (cf. \cite{Taylor}).

Most of the studies of Taylor-Couette flow have been made for just one confined fluid. However, there are situations where it is relevant to consider also the dynamics of two or more fluid placed in the Taylor-Couette geometry.
The stability properties of these have been study by Renardy and Joseph in \cite{renardy}. Numerical simulations indicate that in some parameter regimes circular interfaces separating two fluids become unstable and patterns exhibiting fingering develop (cf. \cite{Hua}, \cite{peng}).

We will be concerned with the study of the Taylor-Couette flows for two fluids, in the case in which the volume filled by one of the fluids is much smaller than the other one. More precisely, if we denote $R_1$, $R_2$ the radii of the internal and external cylinders enclosing the fluids, we will consider flows in which we parametrize the interface in polar coordinates as
$\left\{  r=R_{1}(1+\varepsilon h\left(  \theta,t\right) ) \right\}  $ or $\left\{  r=R_{2}(1-\varepsilon\frac{R_1}{R_2} h\left(  \theta,t\right))  \right\}  $. Notice that $\varepsilon$ measures the thickness of the thin layer of fluid in non-dimensional units.  In this paper we will consider only solutions in which $h(\theta,t)>0$ for all $(\theta,t)$. In particular, we will not consider situations in which the interface has contact lines with the internal cylinder. 

We will restrict ourselves to the case in which the velocity of the fluids satisfies the two-dimensional Navier-Stokes equations. Notice that we ignore the dependence on the perpendicular component $z.$ We assume that the inner cylinder rotates with angular velocity $\omega$, while we keep the outer cylinder at rest. We will set the center of the confining cylinders as the origin of coordinates $O.$ 

Under these assumptions it will be possible to derive, using matched asymptotic expansions (see Section \ref{S3}), a thin film approximation for the evolution of the interface separating both fluids. More precisely, if we write the surface tension in non-dimensional form as $\gamma=\frac{\tilde{\gamma}}{\rho_2R_1\e{3}\omega\e{2}}$,  we obtain that the evolution of the interface separating both fluids satisfies the following PDE:%
\begin{equation}
h_t+\de{}(\frac{h^2}{2})+\de{}{(h^3(\de{}{h}+\partial_\theta^3 h))}=0\label{S1E1}%
\end{equation}
when we consider that $\gamma$ has the form $\gamma\approx\frac{b}{\varepsilon\e{2}}$ for some $b>0$ as $\varepsilon\to\infty$. The constant $b$ has be eliminated from \eqref{S1E1} by means of a trivial rescaling.  In this same Section, we deduce the evolution of the interface when $\gamma$ is much larger that $\frac{1}{\varepsilon\e{2}}$. In that case we have,
\begin{equation}\label{S1E2}
h_t+\de{}{(h^3(\de{}{h}+\partial_\theta^3 h))}=0
\end{equation}

Equations of the same type as those in \eqref{S1E1} and \eqref{S1E2} appear in the study of the motion of thin layers of fluid moving in inclined planes (cf. \cite{benney}, \cite{manneville} and \cite{planosinclinados}), although the detailed form of the nonlinear terms as well as the boundary conditions are different from the ones in this paper. Thin film approximations of free boundary problems for the Stokes equation using matched asymptotics, have been extensively used in the physical and mathematical literature (cf. \cite{Ockendon} and references
therein, \cite{Oron}).  Rigorous derivations of the thin film equation taking as starting point a free boundary problem for the Stokes system in the case of one fluid has been considered in \cite{Gunther}. The analysis of thin film equations in the present of contact lines is well developed research area (cf. \cite{beretta},  \cite{bernis_friedman}, \cite{bernis_peletier}, \cite{bertozzi_pugh}, \cite{ferreira_bernis}, \cite{fischer}, \cite{giacomelli_gnann}, \cite{giacomelli_otto}, \cite{grun}, \cite{otto}). The dynamics of a coupled system for a thin film approximation of the two phase Stokes flow has been considered in \cite{Escher} as well as \cite{Belinchon}. In particular, in \cite{Escher} has been proved that the interfaces converge exponentially to a planar stationary solution in the particular setting considered in that paper.

Section \ref{S4} is devoted to analyse two different kind of steady solutions of \eqref{S1E1} and \eqref{S1E2}. We are interested in constant solutions or in solutions close to constant for equations \eqref{S1E1} and \eqref{S1E2} in the original coordinate system and in the rotating one.

In Section \ref{S5} we study rigorously the stability of the constant solutions of the equations (\ref{S1E1}) and \eqref{S1E2}.  These constant solutions describe circular interfaces which are concentric with the confining cylinders. In Subsection \ref{S5ss1}, in order to study the stability of these solutions for the case with $\gamma\approx\frac{b}{\varepsilon\e{2}}$, we first prove a global existence result (cf. Theorem \ref{thglobal}). In order to study the long time asymptotic of the solutions we try a linearisation argument. It turns out that the resulting linearised problem has two zero eigenvalues, and therefore the stability properties of the constant solutions depend on quadratic and higher order terms. In particular small perturbations of the constant solution do not converge exponentially to zero in general. To deal with this difficulty we use the theory of center manifolds for quasilinear systems in the form developed in \cite{Mielke} and \cite{centermanif}. This allows us to prove that the solutions of (\ref{S1E1}) with initial data close to constant converge to the constant value with an error of order $O\left(  \frac {1}{\sqrt{t}}\right)  $ as $t\rightarrow\infty.$ A more detailed analysis of the solution shows that, for long times, the interface behaves to the leading
order as a circle whose center moves along a spiral towards the origin $O$ as $t\rightarrow\infty$ (cf. Theorem \ref{thestab}). Notice that this asymptotic behaviour of the interface for long times holds for arbitrary choices of the
viscosities and radius of the two viscous fluids. In Subsection \ref{S5ss2}, we deal with the case for $\gamma\gg\frac{1}{\varepsilon\e{2}}$ (cf. Theorem \ref{thglobalgrande}). A direct computation allows us to conclude that solutions can be interpreted as circular interfaces with a center shifted from the origin $O$.

\section{2D Taylor-Couette flow for two incompressible fluids}\label{S2}

We first formulate the equations of the Taylor-Couette flow for two immiscible fluids. The velocity of the fluid is given by Navier-Stokes problem:
\begin{equation*}
\begin{cases}
\rho_1(\tilde{u}\e{1}_t+\tilde{u}\e{1}\cdot\nabla \tilde{u}\e{1})=-\nabla \tilde{p}\e{1}+\mu_1\Delta\tilde{u}\e{1},\nabla\cdot \tilde{u}\e{1}=0\quad\textit{in}\quad \tilde{\Omega}_1(t)\\
\rho_2(\tilde{u}\e{2}_t+\tilde{u}\e{2}\cdot\nabla \tilde{u}\e{2})=-\nabla\tilde{p}\e{2}+\mu_2\Delta\tilde{u}\e{2},\nabla\cdot \tilde{u}\e{2}=0\quad\textit{in}\quad \tilde{\Omega}_2(t)
\end{cases}
\end{equation*}
where $(\tilde{\textbf{x}},\tilde{t})\in\R^{2}\times\R^{+}$, $\tilde{u}\e{i}=(\tilde{u}\e{i}_{1}(\tilde{\textbf{x}},\tilde{t}),\tilde{u}\e{i}_{2}(\tilde{\textbf{x}},\tilde{t}))$ for $i=1,2$ is the incompressible velocity, $\mu_i$ are dynamic viscosity of $\tilde{u}\e{i}$ respectively for $i=1,2$ and $\tilde{p}\e{i}$ are the corresponding pressure in each of the fluids.
\begin{figure}[h]\label{figdom}
\includegraphics[width=90mm]{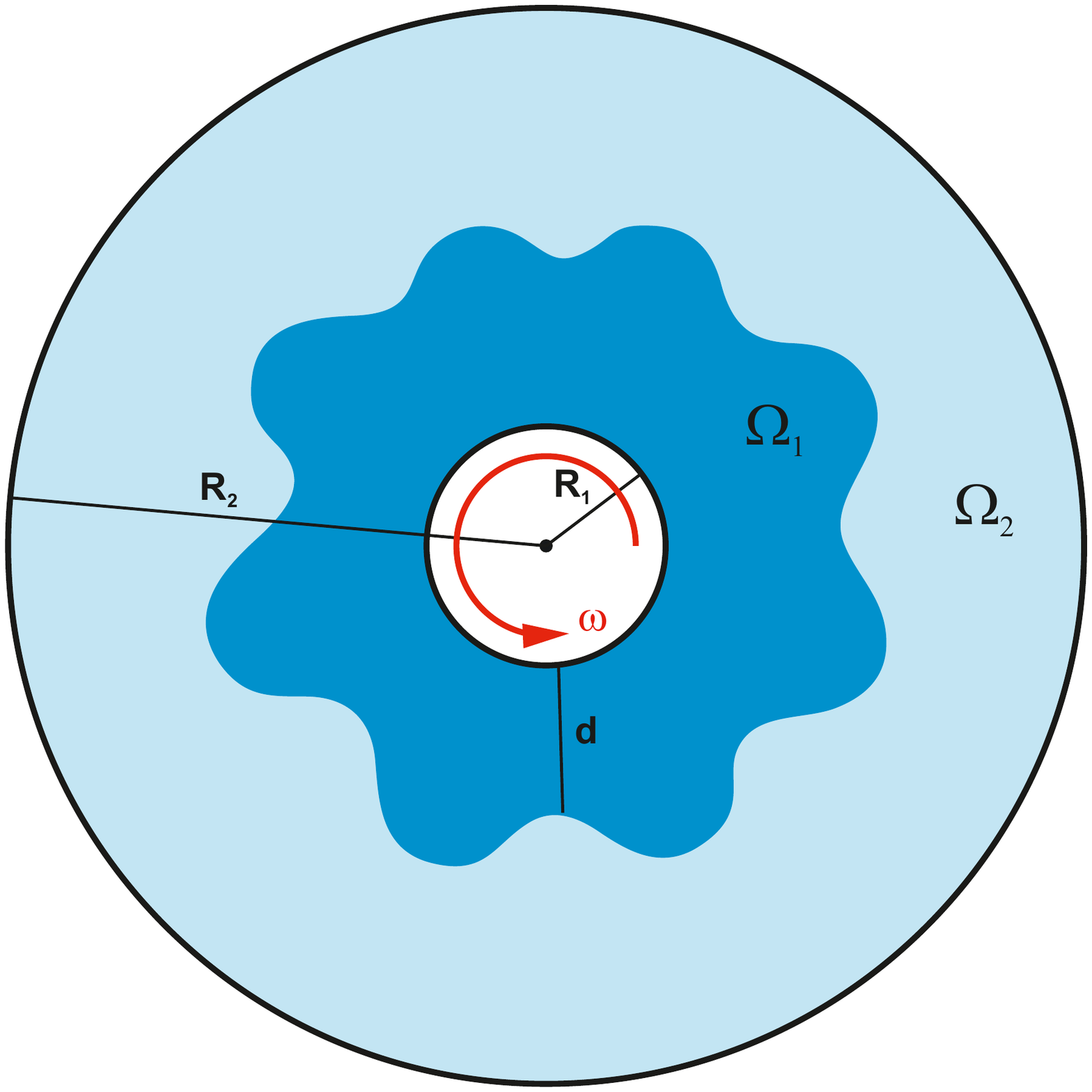}
\caption{Taylor-Couette flow for two fluids}\end{figure}
 As we can see in the Figure \ref{figdom} we use polar coordinates, i.e., $\tilde{\textbf{x}}=(\tilde{x}_1,\tilde{x}_2)=(r\cos\theta,r\sin\theta)$ that 
\begin{displaymath}
\begin{cases}
 \tilde{\Omega}_1(\tilde{t})=\{\tilde{\textbf{x}}\in\R\e{2}:R_1<r<R_1+d h(\theta,\tilde{t})\},\\
 \tilde{\Omega}_2(\tilde{t})=\{\tilde{\textbf{x}}\in\R\e{2}:R_1+d h(\theta,\tilde{t})<r<R_2\}.
\end{cases}
\end{displaymath}
We assume that $h(\theta,\tilde{t})>0$ and $d>0$ is the average high of the fluid 1 which is defined by:
\begin{displaymath}
A_0=\pi\Big((R_1+d)\e{2}-R_1\e{2}\Big)
\end{displaymath}
where $A_0$ is the area filled by the fluid 1.
 
Moreover, the following boundary conditions hold:
\begin{enumerate}
\item $\tilde{u}^1= \omega(\tilde{x}_2,-\tilde{x}_1)$ for $\tilde{\textbf{x}}\in\partial B_{R_1}(0)$ and $\tilde{u}^2=0$ for $\tilde{\textbf{x}}\in\partial B_{R_2}(0)$
\item $\tilde{u}^1\cdot n=\tilde{u}^2\cdot n=V_n$,
\item $\tilde{u}^1\cdot t=\tilde{u}^2\cdot t$,\hspace*{4cm}$\textit{in}\quad\partial\Omega$
\item $t(\tilde{\Sigma}^2-\tilde{\Sigma}^1)\cdot n=0$,
\item $n(\tilde{\Sigma}^2-\tilde{\Sigma}^1)\cdot n=\tilde{\gamma}\tilde{\kappa}$.
\end{enumerate}
Where $n$ is the normal vector to the interface pointing from domain $ \tilde{\Omega}_1$ to $\tilde{\Omega}_2$, $t$ the tangent vector, $\omega$ the angular velocity of the inner cylinder, $V_n$ the normal velocity at the interface, $\tilde{\gamma}$ surface tension, $\tilde{\kappa}$ curvature of the interface and $\tilde{\Sigma}^i=-\tilde{p}\e{i}I+\mu_i(\nabla \tilde{u}\e{i}+(\nabla \tilde{u}\e{i})\e{T})$ represents the stress tensor corresponding each domain $ \tilde{\Omega}^i$.

Now, in order to consider the non-dimensional form of the problem we take:

\begin{align*}
 &\tilde{\textbf{x}}=R_1\textbf{x};\quad \tilde{t}=\frac{1}{\omega}t;\quad \tilde{u}\e{i}=\omega R_1 u\e{i};\quad \tilde{p}\e{i}=\rho_2\omega\e{2} R_1\e{2} p\e{i};\quad\varepsilon=\frac{d}{R_1};\quad \gamma=\frac{\tilde{\gamma}}{\rho_2R_1\e{3}\omega\e{2}};\quad\eta=\frac{R_2}{R_1};\\
 &\zeta=\frac{\rho_1}{\rho_2};\quad Re=\frac{\rho_2\omega R_1\e{2}}{\mu_2}\quad\textit{and}\quad \mu=\frac{\mu_1}{\mu_2}.
\end{align*}

For this case (see Figure \ref{figdomthin}), we have:
\begin{equation}\label{nsnodim}
\begin{cases}
\zeta\Big(u\e{1}_t+u\e{1}\cdot\nabla u\e{1}\Big)=-\nabla p\e{1}+\frac{\mu}{Re}\Delta u\e{1},\nabla\cdot u\e{1}=0\quad\textit{in}\quad \Omega_1(t)\\
u\e{2}_t+u\e{2}\cdot\nabla u\e{2}=-\nabla p\e{2}+\frac{1}{Re}\Delta u\e{2},\nabla\cdot u\e{2}=0\quad\textit{in}\quad \Omega_2(t)
\end{cases}
\end{equation}
where 
\begin{displaymath}
\begin{cases}
\Omega_1(t)=\{\textbf{x}\in\R\e{2}:1<r<1+\varepsilon h(\theta,t)\},\\
\Omega_2(t)=\{\textbf{x}\in\R\e{2}:1+\varepsilon h(\theta,t)<r<\eta\}.
\end{cases}
\end{displaymath}
\begin{figure}[h]\label{figdomthin}
\includegraphics[width=90mm]{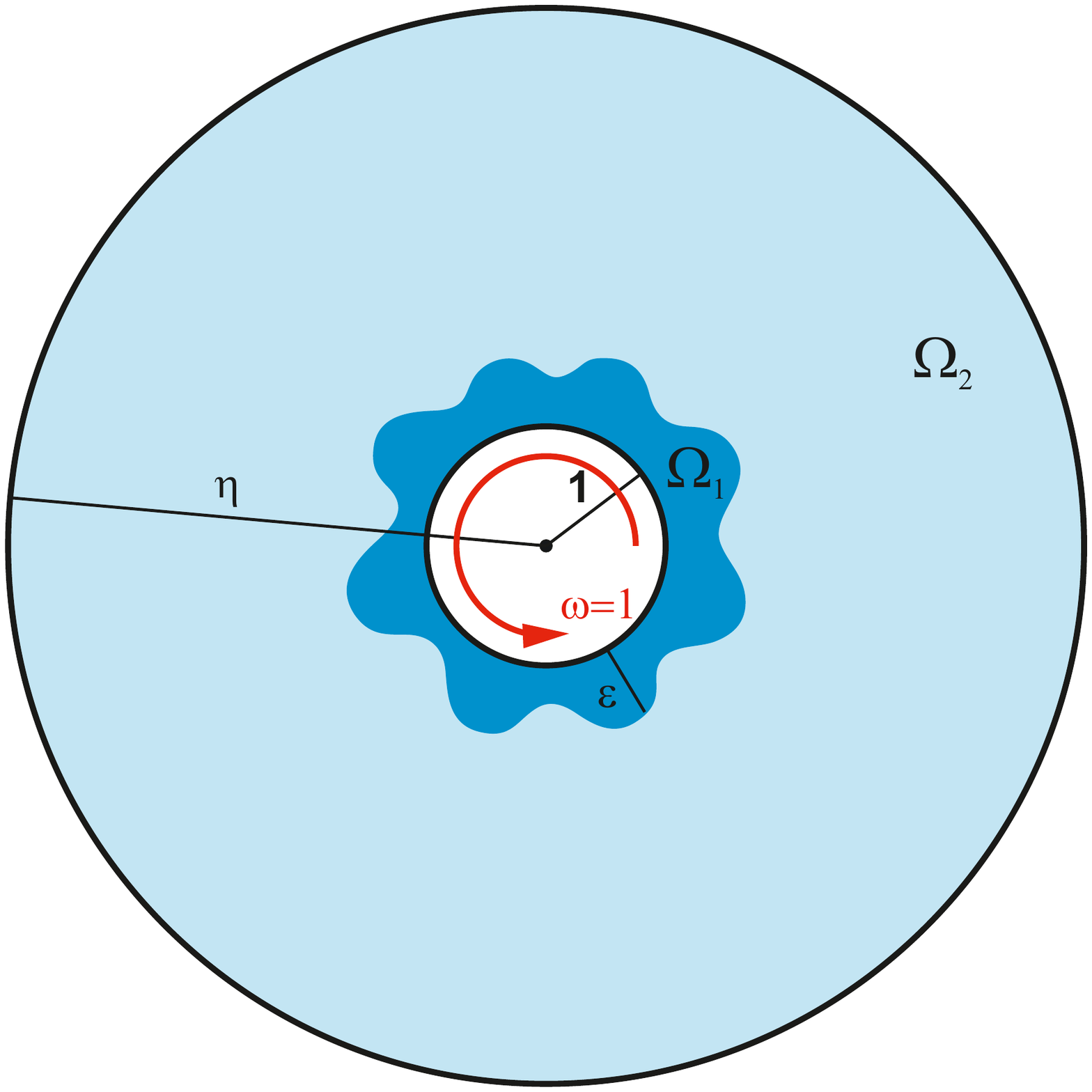}\caption{Taylor-Couette flow with a thin film layer near the internal cylinder (in non-dimensional units)}
\end{figure}

The boundary conditions are as follows:

\begin{align}
&u^1=(x_2,-x_1)\quad\textit{for}\quad\textbf{x}\in\partial B_{1}(0)\quad\textit{and}\quad u^2=0\quad\textit{for}\quad\textbf{x}\in\partial B_{\eta}(0)\label{nscond1}\\
&u^1\cdot n=u^2\cdot n=V_n\quad\textit{in}\quad\partial\Omega\label{nscond2}\\
&u^1\cdot t=u^2\cdot t\quad\textit{in}\quad\partial\Omega\label{nscond3}\\
&t(\Sigma^2-\Sigma^1)\cdot n=0\quad\textit{in}\quad\partial\Omega\label{nscond4}\\
&n(\Sigma^2-\Sigma^1)\cdot n=\gamma\kappa\quad\textit{in}\quad\partial\Omega.\label{nscond5}
\end{align}

Here $n$ is the normal vector to the interface pointing from domain $ \Omega_1$ to $ \Omega_2$, $t$ the tangent vector, $V_n$ the normal velocity at the interface, $\gamma$ is the non-dimensional surface tension, $\kappa$ is the curvature of the interface and $\Sigma^i$ represents the stress tensor corresponding each domain $\Omega^i$, namely:
\begin{equation}\label{sigmas}
\begin{cases}
\Sigma\e{1}=-p\e{1}I+\frac{\mu}{Re}(\nabla u\e{1}+(\nabla u\e{1})\e{T})\\
\Sigma\e{2}=-p\e{2}I+\frac{1}{Re}(\nabla u\e{2}+(\nabla u\e{2})\e{T})
\end{cases}
\end{equation}

The system of equations \eqref{nsnodim}-\eqref{sigmas} describes the 2D Taylor-Couette flow for two incompressible fluids.

In this paper we will consider the case in which one the volume filled by one of the fluids is much smaller than the other, namely, $\varepsilon\to 0\e{+}$ or $\varepsilon\to(\eta-1)\e{-}$. There is a mathematical interesting limit in which the effect of the surface tension and the shear are comparable. This corresponds to taking $\gamma\approx\frac{1}{\varepsilon\e{2}}$ and $\varepsilon\to 0\e{+}$
(alternatively, $\gamma\approx\frac{1}{(\eta-1-\varepsilon)\e{2}}$ for $\varepsilon\to(\eta-1)\e{-}$). In these asymptotic regimes we will be able to use a thin film approximation to describes the form of the interface. Notice that this limit is physically meaningful because if we assume that the fluid 1 is oil and the fluid 2 is water, since $\tilde{\gamma}=0.05$ $N/m$ and $\rho_2=1000$ $kg\cdot m\e{-3}$ then $\gamma=\frac{5\cdot 10\e{-5}m\e{3}s\e{-2}}{R_1\e{3}\omega\e{2}}$. For instance, if we take $R_1=1$ $mm$ and $\omega=1$ $rpm$ we obtain $\gamma=1267$ that can be written as $\frac{1}{\varepsilon\e{2}}$ with $\varepsilon=0.03$. This corresponds to an oil layer of $30$ $\mu m$.

\section{Derivation of a thin film approximations for two-fluids Taylor-Couette flows}\label{S3}

In this section we derive the equation \eqref{S1E1} using formal matched asymptotic expansions. We will assume that the parameters $\zeta$, $(\eta-1)$ and $\mu$ are kept of order one. The Reynolds number $Re$ will be assumed to be of order one but sufficiently small (including zero) in order to ensure that the Taylor instability for the Taylor-Couette flow does not arise (cf. \cite{Taylor}). The issue of the stability of the laminar two-fluid Taylor-Couette flow is an interesting question which deserves further study. The numerical results in \cite{renardy} indicate that the above mention solution is stable for sufficiently small Reynolds numbers.
\subsection{Case $\gamma\approx\frac{b}{\varepsilon\e{2}}$}\label{S3s1}
We consider first the case which the volume fraction of the fluid 1 is much smaller than the one filled by the fluid 2. Under such assumption we will set that the surface tension scales with the non-dimensional thickness $\varepsilon$ as 
\begin{equation}\label{tens}
\gamma\approx\frac{b}{\varepsilon\e{2}}\quad\textit{as}\quad\varepsilon\to 0\quad\textit{and}\quad b>0.
\end{equation}

We rewrite \eqref{nsnodim} in polar coordinates:
\begin{displaymath}
\left\{ \begin{array}{ll}
\dr{(ru\e{1}_r)}+\de{}{u\e{1}_{\theta}}=0\\
\zeta\Big(\partial_t u\e{1}_r+u\e{1}_r\partial_r u\e{1}_r+\frac{u\e{1}_\theta}{r}\de{}{u}\e{1}_r-\frac{(u\e{1}_\theta)\e{2}}{r}\Big)=\frac{\mu}{Re}\Big(\dr{(\frac{1}{r}\dr{(ru\e{1}_r)})}+\frac{\de{2}{u}\e{1}_{r}}{r^2}-\frac{2\de{}{u}\e{1}_{\theta}}{r^2}\Big)-\dr{p\e{1}}\\
\zeta\Big(\partial_t u\e{1}_\theta+u\e{1}_r\partial_r u\e{1}_\theta+\frac{u\e{1}_\theta}{r}\de{}{u\e{1}_\theta}-\frac{u\e{1}_r u\e{1}_\theta}{r}\Big)=\frac{\mu}{Re}\Big(\dr{(\frac{1}{r}\dr{(ru\e{1}_\theta)})}+\frac{\de{2}{u\e{1}_{\theta}}}{r^2}+\frac{2\de{}{u\e{1}_{r}}}{r^2}\Big)-\frac{1}{r}\de{}{p\e{1}}
\end{array} \right.
\end{displaymath}

\begin{displaymath}
\left\{ \begin{array}{ll}
\dr{(ru\e{2}_r)}+\de{}{u\e{2}_{\theta}}=0\\
\partial_t u\e{2}_r+u\e{2}_r\partial_r u\e{2}_r+\frac{u\e{2}_\theta}{r}\de{}{u\e{2}_r}-\frac{(u\e{2}_\theta)\e{2}}{r}=\frac{1}{Re}\Big(\dr{(\frac{1}{r}\dr{(ru\e{2}_r)})}+\frac{\de{2}{u\e{2}_{r}}}{r^2}-\frac{2\de{}{u\e{2}_{\theta}}}{r^2}\Big)-\dr{p\e{2}}\\
\partial_t u\e{2}_\theta+u\e{2}_r\partial_r u\e{2}_\theta+\frac{u\e{2}_\theta}{r}\de{}{u\e{2}_\theta}-\frac{u\e{2}_r u\e{2}_\theta}{r}=\frac{1}{Re}\Big(\dr{(\frac{1}{r}\dr{(ru\e{2}_\theta)})}+\frac{\de{2}{u\e{2}_{\theta}}}{r^2}+\frac{2\de{}{u\e{2}_{r}}}{r^2}\Big)-\frac{1}{r}\de{}{p\e{2}}
\end{array} \right.
\end{displaymath}

In order to get the evolution equation in the limit we are going to use asymptotics. To make the computations easier we make the following change of variables: 
\begin{displaymath}
u\e{i}_r=\varepsilon^2 w\e{i}_{\xi};\quad u\e{i}_\theta - 1=\varepsilon w\e{i}_\theta;\quad p\e{i}=\frac{1}{\varepsilon}P\e{i}\quad\textit{for}\quad i=1,2\quad\textit{and}\quad \xi=\frac{r-1}{\varepsilon}.
\end{displaymath}
Using the above changes the equations are the following:

\begin{displaymath}
\left\{ \begin{array}{ll}
\dx{ w\e{1}_\xi}+\varepsilon\dx{(\xi w\e{1}_\xi)}+\de{}{ w\e{1}_{\theta}}=0\\
\zeta\Big(\varepsilon\e{4}\partial_t w\e{1}_\xi+\varepsilon\e{5}w\e{1}_\xi\partial_\xi w\e{1}_\xi+\frac{\varepsilon\e{4}}{1+\varepsilon\xi}(\varepsilon w\e{1}_\theta +1)\de{}{w\e{1}_\xi}-\frac{\varepsilon\e{2}}{1+\varepsilon\xi}(\varepsilon w\e{1}_\theta
+1)\e{2}\Big)=\\
\frac{\mu}{Re}\Big(\varepsilon^2\dx{(\frac{1}{1+\varepsilon\xi}\dx{((1+\varepsilon\xi) w\e{1}_\xi)})}+\frac{\varepsilon^4}{(1+\varepsilon\xi)^2}\de{2}{ w\e{1}_\xi}-\frac{2\varepsilon^3}{(1+\varepsilon\xi)^2}\de{}{ w\e{1}_\theta}\Big)-\dx{P\e{1}}\\
\zeta\Big(\varepsilon\e{2}\partial_t w\e{1}_\xi+\varepsilon\e{3}w\e{1}_\xi\partial_\xi w\e{1}_\theta+\frac{\varepsilon\e{2}}{1+\varepsilon\xi}(\varepsilon w\e{1}_\theta +1)\de{}{w\e{1}_\theta}-\frac{\varepsilon\e{3}}{1+\varepsilon\xi}w\e{1}_\xi(\varepsilon w\e{1}_\theta
+1)\Big)=\\
\frac{\mu}{Re}\Big(\dx{(\frac{1}{1+\varepsilon\xi}\dx{((1+\varepsilon\xi) w\e{1}_\theta)})}-\frac{\varepsilon}{(1+\varepsilon\xi)^2}+\frac{\varepsilon^2}{(1+\varepsilon\xi)^2}\de{2}{ w\e{1}_\theta}+\frac{2\varepsilon^3}{(1+\varepsilon\xi)^2}\de{}{ w\e{1}_\xi}\Big)-\frac{1}{(1
+\varepsilon\xi)}\de{}{P\e{1}}
\end{array} \right.
\end{displaymath}

\begin{displaymath}
\left\{ \begin{array}{ll}
\dx{ w\e{2}_\xi}+\varepsilon\dx{(\xi w\e{2}_\xi)}+\de{}{ w\e{2}_{\theta}}=0\\
\varepsilon\e{4}\partial_t w\e{2}_\xi+\varepsilon\e{5}w\e{2}_\xi\partial_\xi w\e{2}_\xi+\frac{\varepsilon\e{4}}{1+\varepsilon\xi}(\varepsilon w\e{2}_\theta +1)\de{}{w\e{2}_\xi}-\frac{\varepsilon\e{2}}{1+\varepsilon\xi}(\varepsilon w\e{2}_\theta
+1)\e{2}=\\
\frac{1}{Re}\Big(\varepsilon^2\dx{(\frac{1}{1+\varepsilon\xi}\dx{((1+\varepsilon\xi) w\e{2}_\xi)})}+\frac{\varepsilon^4}{(1+\varepsilon\xi)^2}\de{2}{ w\e{2}_\xi}-\frac{2\varepsilon^3}{(1+\varepsilon\xi)^2}\de{}{ w\e{2}_\theta}\Big)-\dx{P\e{2}}\\
\varepsilon\e{2}\partial_t w\e{2}_\xi+\varepsilon\e{3}w\e{2}_\xi\partial_\xi w\e{2}_\theta+\frac{\varepsilon\e{2}}{1+\varepsilon\xi}(\varepsilon w\e{2}_\theta +1)\de{}{w\e{2}_\theta}-\frac{\varepsilon\e{3}}{1+\varepsilon\xi}w\e{2}_\xi(\varepsilon w\e{2}_\theta
+1)=\\
\frac{1}{Re}\Big(\dx{(\frac{1}{1+\varepsilon\xi}\dx{((1+\varepsilon\xi) w\e{2}_\theta)})}-\frac{\varepsilon}{(1+\varepsilon\xi)^2}+\frac{\varepsilon^2}{(1+\varepsilon\xi)^2}\de{2}{ w\e{2}_\theta}+\frac{2\varepsilon^3}{(1+\varepsilon\xi)^2}\de{}{ w\e{2}_\xi}\Big)-\frac{1}{(1
+\varepsilon\xi)}\de{}{P\e{2}}
\end{array} \right.
\end{displaymath}

and the boundary conditions are 
\begin{align}
&w_\xi^1(0,\theta)=0\quad\textit{and}\quad w_\theta^1(0,\theta)=0\label{cond1}\\
&(1+\varepsilon h)\varepsilon w_\xi^1-\de{}{h}(\varepsilon w_\theta^1+1)=(1+\varepsilon h)\varepsilon w_\xi^2-\de{}{h}(\varepsilon w_\theta^2+1)=(1+\varepsilon h)\partial_t h\label{cond2}\\
&\varepsilon^3\de{} h w_\xi^1+(1+\varepsilon h)(\varepsilon w_\theta^1+1)=\varepsilon^3\de{} h w_\xi^2+(1+\varepsilon h)(\varepsilon w_\theta^2+1)\label{cond3}\\
&(\sigma_{\xi\xi}]^2_1-\sigma_{\theta\theta}]^2_1)(1+\varepsilon h)\varepsilon\de{}{h}+\sigma_{\xi\theta}]^2_1((1+\varepsilon h)^2-\varepsilon^2\de{}{h}^2)=0\label{cond4}\\
&\sigma_{\xi\xi}]^2_1(1+\varepsilon h)^2+\varepsilon^2\de{}{h}^2\sigma_{\theta\theta}]^2_1-2(1+\varepsilon h)\varepsilon\de{}{h}\sigma_{\xi\theta}]^2_1=((1+\varepsilon h)^2+\varepsilon^2\de{}{h}^2)\gamma\kappa.\label{cond5}
\end{align}

where $\sigma_{\xi\xi}$, $\sigma_{\xi\theta}$ and $\sigma_{\theta\theta}$ are the coefficients  of the stress tensor:
\begin{displaymath}
\sigma_{\xi\xi}\e{i}=-\frac{1}{\varepsilon}P\e{i}+\frac{2c\e{i}\varepsilon}{Re}\dx{ w\e{i}_\xi},
\end{displaymath}
\begin{displaymath}
\sigma_{\xi\theta}\e{i}=\frac{c\e{i}}{Re}(\frac{\varepsilon^2}{1+\varepsilon h}\de{}{ w\e{i}_\xi}+\dx{ w\e{i}_\theta}-\frac{\varepsilon w\e{i}_\theta +1}{1+\varepsilon h}),
\end{displaymath}
\begin{displaymath}
\sigma_{\theta\theta}\e{i}=-\frac{1}{\varepsilon}P\e{i}+\frac{2c\e{i}}{Re(1+\varepsilon h)}(\varepsilon\de{}{ w\e{i}_\theta}+\varepsilon^2 w\e{i}_\xi),
\end{displaymath}
with 
\begin{displaymath}
c\e{i}=\begin{cases}
\mu\quad\textit{for}\quad i=1\\
1\quad\textit{for}\quad i=2
\end{cases}
\end{displaymath}
Furthermore, we know that the curvature $\kappa$ of the interface $r=1+\varepsilon h(\theta,t)$ is given by:
\begin{equation}\label{curvatura}
\kappa=\frac{2\varepsilon\e{2}(\de{}{h})\e{2}-\varepsilon(1+\varepsilon h)\de{2}{h}+(1+\varepsilon h)\e{2}}{\Big((1+\varepsilon h)\e{2}+\varepsilon\e{2}(\de{}{h})\e{2}\Big)\e{\frac{3}{2}}}
\end{equation}
If we keep only the terms of order 1, we have:
\begin{equation}\label{sistorden1}
\left\{ \begin{array}{ll}
\dx{ w\e{i}_\xi}+\de{}{ w\e{i}_{\theta}}=0\quad\textit{in}\quad\Omega\e{i}(t)\quad\textit{for}\quad i=1,2; \\
\dx{P\e{i}}=0\quad\textit{in}\quad\Omega\e{i}(t)\quad\textit{for}\quad i=1,2;\\
\frac{c\e{i}}{Re}\dxx{ w\e{i}_\theta}-\de{}{P\e{i}}=0\quad\textit{in}\quad\Omega\e{i}(t).
\end{array} \right.
\end{equation}
Thus we know that 
\begin{displaymath}
 w_\theta^i(\xi,\theta)=\frac{Re}{2c\e{i}}\de{}{P\e{i}(\theta)}\xi^2+A^i(\theta)\xi+B^i(\theta);
\end{displaymath}

\begin{displaymath}
 w_\xi^i(\xi,\theta)=-\frac{Re}{6c\e{i}}\de{2}{P\e{i}(\theta)}\xi^3-\de{}{A^i(\theta)}\frac{\xi^2}{2}-\de{}{B^i(\theta)}\xi+C^i(\theta);
\end{displaymath}

Since the volume filled by fluid 1 is very small, we can expect the velocity of the fluid 2 to be a small perturbation of the Taylor-Couette flow for a single fluid confined between two cylinders, i.e. $u_\theta(r)=D_1 r+\frac{D_2}{r}$ where $D_1=\frac{-1}{\eta\e{2}-1}$ and $D_2=\frac{\eta\e{2}}{\eta\e{2}-1}$.

 If we approximate the Taylor-Couette flow by its Taylor polynomial around $r=1$,
\begin{displaymath}
u_\theta(r)\sim 1+(D_1-D_2)(r-1)+2D_2(r-1)^2+\cdots
\end{displaymath}
Then if we do the change of variables, we get
\begin{displaymath}
 w_\theta^2\sim(D_1-D_2)\xi +2D_2\varepsilon\xi^2+\cdots
\end{displaymath}
Therefore, doing the matching we can deduce that $A^2(\theta)=(D_1 -D_2)=-\frac{1+\eta\e{2}}{\eta\e{2}-1}$.
 Moreover, we can do the matching with the pressure. In the Taylor-Couette flow, the pressure is constant, then in the leading order the pressure will be a constant too, i.e., $P^2=const$.

Now using boundary condition \eqref{cond1}, we get $B^1(\theta)=C^1(\theta)=0$.

If we consider condition \eqref{cond4}, we have $\sigma_{\xi\theta}]^2_1=0$, that is in the leading order, $$\frac{1}{Re}(\dx w_\theta^2 - 1)=\frac{\mu}{Re}(\dx w_\theta^1 -1).$$ From this equation and taking in to account that $\de{}{P^2}=0$ we have: 
\begin{displaymath}
A^1(\theta)=-\frac{Re}{\mu}\de{}{P^1}h-\frac{1}{\mu}\frac{1+\eta\e{2}}{\eta\e{2}-1}+\frac{\mu -1}{\mu}.
\end{displaymath}
Using condition \eqref{cond3}, we know that $ w_\theta^1= w_\theta^2$. Then, if we use the expression of $A^1(\theta)$, 
\begin{displaymath}
B^2(\theta)=-\frac{Re}{2\mu}\de{}{P^1}h^2+(1-\frac{1}{\mu})\frac{1+\eta\e{2}}{\eta\e{2}-1}h+\frac{\mu -1}{\mu}h
\end{displaymath}
In order to compute $C^2(\theta)$ we use condition \eqref{cond2}, that to the leading order becomes $ w_\xi^1 -\de{}{h} w_\theta^1= w_\xi^2 -\de{}{h} w_\theta^2$. Thus,
\begin{displaymath}
C^2(\theta)=-\frac{Re}{6\mu}\de{}(h^3\de{}{P^1})-\frac{1}{2}\de{}(A^1(\theta)h^2)-\frac{1+\eta\e{2}}{\eta\e{2}-1}\de{}{(\frac{h^2}{2})}+\de{}(B^2(\theta)h).
\end{displaymath}
If we consider condition \eqref{cond5}, we have $\sigma_{\xi\xi}]^2_1=\gamma\kappa$ therefore, $P^1-P^2=\varepsilon\gamma\kappa$. Since $P^2=cte$ and using \eqref{curvatura} we can approximate $\kappa\approx 1-\varepsilon(h+\de{2}{h})$, we have $$\de{}{P^1}=\varepsilon\gamma\de{}{\kappa}=-\gamma\varepsilon^2(\de{}{h}+\partial_\theta^3{h}).$$ 

Using condition \eqref{cond2} again we can see that evolution equation of the interface is
\begin{displaymath}
\partial_t h +(1-\varepsilon h)\de{}{h}-\varepsilon( w_\xi^1 -\de{}{h} w_\theta^1)=0.
\end{displaymath}
Substituting all above terms in the equation we have:
\begin{equation}\label{eqevol1}
\partial_t h+(1-\varepsilon h)\de{}{h}+\varepsilon\bigg(\frac{\gamma\varepsilon^2Re}{3\mu}\de{}{(h^3(\de{}{h}+\partial_\theta^3 h))}-\frac{1}{\mu}\frac{1+\eta\e{2}}{\eta\e{2}-1}\de{}(\frac{h^2}{2})+\frac{\mu-1}{\mu}\de{}(\frac{h^2}{2})\bigg)=0
\end{equation}

This equation can be simplified changing to a coordinate system that rotates at speed 1, in addition changing the unit of time and rescaling the function $h$. We also change the angle from $\theta$ to $-\theta$ in order to get a positive sign in the first order term. Taking $\gamma=\frac{b}{\varepsilon^2}$ and using the change of variables: 
\begin{equation}\label{rescale}
\begin{cases}
h(\theta,t)=\lambda\bar{h}(\bar{\theta},\bar{t});\\
\bar{\theta}=-\theta+t,\\
\lambda=\sqrt{\frac{6\eta\e{2}}{Re(\eta\e{2}-1)b}},\\
\bar{t}=\frac{2\eta\e{2}\lambda\varepsilon}{\mu(\eta\e{2}-1)}t,
\end{cases}
\end{equation}
equation \eqref{eqevol1} becomes 
\begin{equation*}
\bar{h}_{\bar{t}}+\partial_{\bar{\theta}}(\frac{\bar{h}^2}{2})+\partial_{\bar{\theta}}(\bar{h}^3(\partial_{\bar{\theta}}\bar{h}+\partial_{\bar{\theta}}^3 \bar{h}))=0.
\end{equation*}
In order to simplify the notation, we will remove the bars. Therefore, we will consider the following equation:
\begin{equation}\label{equevol}
h_t+\de{}(\frac{h^2}{2})+\de{}{(h^3(\de{}{h}+\partial_\theta^3 h))}=0.
\end{equation}
\begin{figure}[h]\label{thinfilmfuera}
\includegraphics[width=90mm]{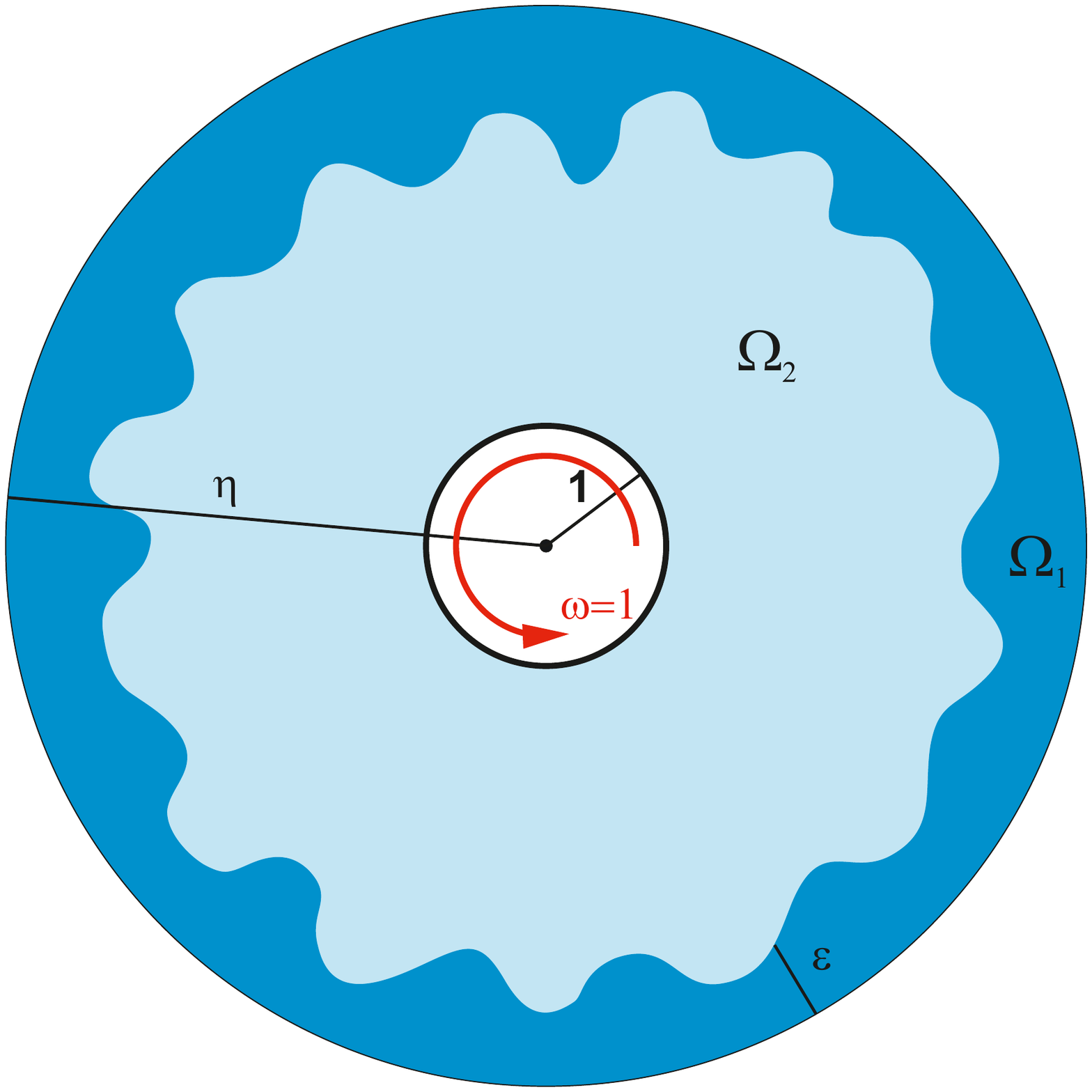}\caption{Taylor-Couette flow with a thin film layer near the external cylinder (in non-dimensional units)}
\end{figure}
\begin{nota}\label{fluidoexterno}
Other possible scenario is the case in which the thin fluid is near the external cylinder. We assume that the fluid 1 is the one filling the small volume fraction (see Figure \ref{thinfilmfuera}). The interface is parametrized by $r=\eta-\varepsilon h(\theta)$ and its curvature is $\kappa\approx\frac{1}{\eta}+\frac{\varepsilon}{\eta}(h+\de{2}{h})$. The resulting evolution equation of the free boundary is similar to one obtained in the previous case.
In order to derive this equation we use the following rescaling:
\begin{displaymath}
u_r\e{i}=\varepsilon^2 w\e{i}_{\xi};\quad u\e{i}_\theta =\varepsilon w\e{i}_\theta;\quad p\e{i}=\frac{1}{\varepsilon}P\e{i};\quad\textit{and}\quad\xi=\frac{r-\eta}{\varepsilon}.
\end{displaymath}
Then, using again polar coordinates, the equations for the fluid velocity \eqref{nsnodim} become:
\begin{displaymath}
\left\{ \begin{array}{ll}
\eta\dx{ w\e{2}_\xi}+\varepsilon\dx{(\xi w\e{2}_\xi)}+\de{}{ w\e{2}_{\theta}}=0\\
\zeta\Big(\varepsilon\e{4}\partial_t w\e{2}_\xi+\varepsilon\e{5}w\e{2}_\xi\partial_\xi w\e{2}_\xi+\frac{\varepsilon\e{5}}{\eta+\varepsilon\xi}w\e{2}_\theta \de{}{w\e{2}_\xi}-\frac{\varepsilon\e{3}}{\eta+\varepsilon\xi}(w\e{2}_\theta
)\e{2}\Big)=\\
\frac{\mu}{Re}\Big(\varepsilon^2\dx{(\frac{1}{\eta+\varepsilon\xi}\dx{((\eta+\varepsilon\xi) w\e{2}_\xi)})}+\frac{\varepsilon^4}{(\eta+\varepsilon\xi)^2}\de{2}{ w\e{2}_\xi}-\frac{2\varepsilon^3}{(\eta+\varepsilon\xi)^2}\de{}{ w\e{2}_\theta}\Big)-\dx{P\e{2}}\\
\zeta\Big(\varepsilon\e{2}\partial_t w\e{2}_\xi+\varepsilon\e{3}w\e{1}_\xi\partial_\xi w\e{2}_\theta+\frac{\varepsilon\e{3}}{\eta+\varepsilon\xi}w\e{2}_\theta \de{}{w\e{2}_\theta}-\frac{\varepsilon\e{4}}{\eta+\varepsilon\xi}w\e{2}_\xi w\e{2}_\theta
\Big)=\\
\frac{\mu}{Re}\Big(\dx{(\frac{1}{\eta+\varepsilon\xi}\dx{((\eta+\varepsilon\xi) w\e{2}_\theta)})}+\frac{\varepsilon^2}{(\eta+\varepsilon\xi)^2}\de{2}{ w\e{2}_\theta}+\frac{2\varepsilon^3}{(\eta+\varepsilon\xi)^2}\de{}{ w\e{2}_\xi}\Big)-\frac{1}{(\eta
+\varepsilon\xi)}\de{}{P\e{2}}
\end{array} \right.
\end{displaymath}
\begin{displaymath}
\left\{ \begin{array}{ll}
\eta\dx{ w\e{1}_\xi}+\varepsilon\dx{(\xi w\e{1}_\xi)}+\de{}{ w\e{1}_{\theta}}=0\\
\varepsilon\e{4}\partial_t w\e{1}_\xi+\varepsilon\e{5}w\e{1}_\xi\partial_\xi w\e{1}_\xi+\frac{\varepsilon\e{5}}{\eta+\varepsilon\xi}w\e{1}_\theta \de{}{w\e{1}_\xi}-\frac{\varepsilon\e{3}}{\eta+\varepsilon\xi}(w\e{1}_\theta
)\e{2}=\\
\frac{1}{Re}\Big(\varepsilon^2\dx{(\frac{1}{\eta+\varepsilon\xi}\dx{((\eta+\varepsilon\xi) w\e{1}_\xi)})}+\frac{\varepsilon^4}{(\eta+\varepsilon\xi)^2}\de{2}{ w\e{1}_\xi}-\frac{2\varepsilon^3}{(\eta+\varepsilon\xi)^2}\de{}{ w\e{1}_\theta}\Big)-\dx{P\e{1}}\\
\varepsilon\e{2}\partial_t w\e{1}_\xi+\varepsilon\e{3}w\e{1}_\xi\partial_\xi w\e{1}_\theta+\frac{\varepsilon\e{3}}{\eta+\varepsilon\xi}w\e{1}_\theta \de{}{w\e{1}_\theta}-\frac{\varepsilon\e{4}}{\eta+\varepsilon\xi}w\e{1}_\xi w\e{1}_\theta
=\\
\frac{1}{Re}\Big(\dx{(\frac{1}{\eta+\varepsilon\xi}\dx{((\eta+\varepsilon\xi) w\e{1}_\theta)})}+\frac{\varepsilon^2}{(\eta+\varepsilon\xi)^2}\de{2}{ w\e{1}_\theta}+\frac{2\varepsilon^3}{(\eta+\varepsilon\xi)^2}\de{}{ w\e{1}_\xi}\Big)-\frac{1}{(\eta
+\varepsilon\xi)}\de{}{P\e{1}}
\end{array} \right.
\end{displaymath}
and the boundary conditions \eqref{nscond1}-\eqref{nscond5} are given by:
\begin{displaymath}
\begin{cases}
w_\xi^1(0,\theta)=0\quad\textit{and}\quad w_\theta^1(0,\theta)=0,\\
(\eta-\varepsilon h)\varepsilon w_\xi^1+\de{}{h}(\varepsilon w_\theta^1+1)=(\eta-\varepsilon h)\varepsilon w_\xi^2+\de{}{h}(\varepsilon w_\theta^2+1)=(\eta-\varepsilon h)\partial_t h\quad\textit{for}\quad\xi=-h,\\
-\varepsilon^3\de{} h w_\xi^1+(\eta-\varepsilon h)(\varepsilon w_\theta^1+1)=-\varepsilon^3\de{} h w_\xi^2+(\eta-\varepsilon h)(\varepsilon w_\theta^2+1)\quad\textit{for}\quad\xi=-h,\\
(\sigma_{\xi\xi}]^2_1-\sigma_{\theta\theta}]^2_1)(\eta-\varepsilon h)\varepsilon\de{}{h}+\sigma_{\xi\theta}]^2_1(\varepsilon^2\de{}{h}^2-(\eta-\varepsilon h)^2)=0\quad\textit{for}\quad\xi=-h,\\
\sigma_{\xi\xi}]^2_1(\eta-\varepsilon h)^2+\varepsilon^2\de{}{h}^2\sigma_{\theta\theta}]^2_1+2(\eta-\varepsilon h)\varepsilon\de{}{h}\sigma_{\xi\theta}]^2_1=((\eta-\varepsilon h)^2+\varepsilon^2\de{}{h}^2)\gamma\kappa\quad\textit{for}\quad\xi=-h.
\end{cases}
\end{displaymath}

where $\sigma_{\xi\xi}$, $\sigma_{\xi\theta}$ and $\sigma_{\theta\theta}$ are the coefficients  of the stress tensor:
\begin{displaymath}
\sigma_{\xi\xi}\e{i}=-\frac{1}{\varepsilon}P\e{i}+\frac{2c\e{i}\varepsilon}{Re}\dx{ w\e{i}_\xi},
\end{displaymath}
\begin{displaymath}
\sigma_{\xi\theta}\e{i}=\frac{c\e{i}}{Re}(\frac{\varepsilon^2}{\eta-\varepsilon h}\de{}{ w\e{i}_\xi}+\dx{ w\e{i}_\theta}-\frac{\varepsilon w\e{i}_\theta +1}{\eta-\varepsilon h}),
\end{displaymath}
\begin{displaymath}
\sigma_{\theta\theta}\e{i}=-\frac{1}{\varepsilon}P\e{i}+\frac{2c\e{i}}{Re(\eta-\varepsilon h)}(\varepsilon\de{}{ w\e{i}_\theta}+\varepsilon^2 w\e{i}_\xi),
\end{displaymath}
with 
\begin{displaymath}
c\e{i}=\begin{cases}
\mu\quad\textit{for}\quad i=2\\
1\quad\textit{for}\quad i=1
\end{cases}
\end{displaymath}
Collecting the terms of order 1, we obtain a system similar to \eqref{sistorden1}:
\begin{displaymath}
\left\{ \begin{array}{ll}
\eta\dx{ w\e{i}_\xi}+\de{}{ w\e{i}_{\theta}}=0\\
\dx{P\e{i}}=0\\
\dxx{ w\e{i}_\theta}-\frac{Re}{\eta c\e{i}}\de{}{P\e{i}}=0
\end{array} \right.
\end{displaymath}
where the only difference is the onset of the non-dimensional radius $\eta$ in the last equation. 
Arguing then similarly as in the derivation of \eqref{eqevol1}, we arrive at:
\begin{displaymath}
\partial_t h-(\frac{1}{\eta}+\frac{\varepsilon h}{\eta\e{2}})\de{}{h}+\varepsilon\bigg(\frac{\gamma\varepsilon^2Re}{3\eta\e{3}}\de{}{(h^3(\de{}{h}+\partial_\theta^3 h))}+\frac{(1-\mu)(\eta\e{2}-1)-\mu\eta(1+\eta\e{2})}{\eta\e{2}(\eta\e{2}-1)}\de{}(\frac{h^2}{2})\bigg)=0
\end{displaymath}
Taking $\gamma=\frac{b}{\varepsilon\e{2}}$ and making the change of variables 
\begin{equation}\label{reescalefuera}
\begin{cases}
h(\theta,t)=\lambda\bar{h}(\bar{\theta},\bar{t});\\
\bar{\theta}=-\theta-\frac{t}{\eta},\\
\bar{t}=\frac{\mu\eta(1+\eta\e{2})+\mu(\eta\e{2}-1)}{\eta\e{2}(\eta\e{2}-1)}\varepsilon\lambda t,\\
\lambda=\sqrt{\frac{3(\mu\eta(1+\eta\e{2})+\mu(\eta\e{2}-1))}{b\eta Re(\eta\e{2}-1)}}
\end{cases}
\end{equation}
and removing again the bars, we obtain \eqref{equevol}

\end{nota}
\subsection{Case $\gamma\gg\frac{1}{\varepsilon\e{2}}$}
If we consider the case in which the non-dimensional surface tension is much larger than $\frac{1}{\varepsilon\e{2}}$ it is natural to use, instead of \eqref{rescale} the following change of variables:
\begin{equation}\label{reescalegrande}
\begin{cases}
h(\theta,t)=\lambda\bar{h}(\bar{\theta},\bar{t});\\
\bar{\theta}=-\theta+t,\\
\lambda=\sqrt{\frac{6\eta\e{2}}{Re(\eta\e{2}-1)}},\\
\bar{t}=\frac{2\eta\e{2}\lambda\varepsilon\e{3}\gamma}{\mu(\eta\e{2}-1)}t,
\end{cases}
\end{equation}

Thus, equation \eqref{eqevol1} becomes:

\begin{equation}\label{eqevtsgrande1}
h_t+\frac{1}{\gamma\varepsilon\e{2}}\de{}{(\frac{h\e{2}}{2})}+\de{}{(h^3(\de{}{h}+\partial_\theta^3 h))}=0,
\end{equation}
that, in the limit when $\varepsilon\to 0$ formally yields:
\begin{equation}\label{eqevtsgrande}
h_t+\de{}{(h^3(\de{}{h}+\partial_\theta^3 h))}=0.
\end{equation}
\section{Stationary solutions and travelling waves of the equations \eqref{equevol} and \eqref{eqevtsgrande}}\label{S4}

The steady states of \eqref{equevol} are the solutions of the ODEs:
\begin{equation}\label{ode1}
\frac{h\e{2}}{2}+h^3(\de{}{h}+\partial_\theta^3 h)=J,\quad\textit{for some}\quad J\in\R
\end{equation}

and those of \eqref{eqevtsgrande} are the solutions of:
\begin{equation}\label{ode2}
h^3(\de{}{h}+\partial_\theta^3 h)=J,\quad\textit{for some}\quad J\in\R.
\end{equation}
In both cases we must have in addition 
\begin{equation}\label{odebc}
h(\theta+2\pi)=h(\theta)\quad\textit{for}\quad\theta\in\R.
\end{equation}
The parameter $J$ in \eqref{ode1} and \eqref{ode2}, can be interpreted as the flux of fluid 1 through any radius of the external cylinder. In the case of the steady states this flux is the same for every radius.

Travelling wave solutions of \eqref{equevol} or \eqref{eqevtsgrande} are solutions with the form 
\begin{equation}\label{travelscale}
h(\theta,t)=h(\theta-ct)\quad\textit{for some}\quad c\in\R
\end{equation}
Therefore, in the case of equation \eqref{equevol}, $h$ solves 
\begin{equation}\label{travel1}
-ch+\frac{h\e{2}}{2}+h^3(\de{}{h}+\partial_\theta^3 h)=J,\quad\textit{for some}\quad J\in\R
\end{equation}
and in the case of \eqref{eqevtsgrande},
\begin{equation}\label{travel2}
-ch+h^3(\de{}{h}+\partial_\theta^3 h)=J,\quad\textit{for some}\quad J\in\R.
\end{equation}
In both cases $h$ satisfies \eqref{odebc}.

From now on, we will use the following notation. We will denote as $H\e{k}(\T)$ for $k\in\{0,1,2,...\}$ the closure in $H\e{k}(0,2\pi)$ of the restriction of the functions in $C\e{\infty}(\R)$ satisfying \eqref{odebc}. We denote as $L\e{\infty}(\T)$ the space $L\e{\infty}(0,2\pi)$. 

Moreover, we will assume that the Hilbert spaces $H\e{k}(\T)$, with their natural scalar product, are always real spaces of real functions. However, in order to simplify the notation we will consider then as closed subspaces of the complex Hilbert spaces $H\e{k}(\T;\C)$. In particular we can represent any $f\in H\e{k}(\T)$ as 
\begin{equation}\label{espacioswzeta}
f(\theta)=\sum_{n=-\infty}\e{\infty}a_n e\e{in\theta}\quad\textit{with}\quad a_n=\overline{a_n}.
\end{equation}

We will need also the homogeneous Sobolev spaces $\dot{H}\e{k}(\T)$ which are the set of functions $f\in H\e{k}(\T)$ such that $\int_{\T}f=0$.

Our goal is to study the steady state problems $[$\eqref{ode1}, \eqref{odebc}$]$, $[$\eqref{ode2}, \eqref{odebc}$]$, $[$\eqref{travel1}, \eqref{odebc}$]$ and $[$\eqref{travel2}, \eqref{odebc}$]$. We first remark that each of these problems can be understood in two natural different ways: 
\begin{enumerate}
\item Existence and uniqueness for these problem for a given value of $J$ (cf. Proposition \ref{prop} and Proposition \ref{proptravel}).
\item To obtain solutions for each of the problems for some $J$ (whose determination is part of the problem) assuming that $\int_{\T}hd\theta$ is given (cf. Proposition \ref{prop3}).
\end{enumerate}  
The following propositions we will address these two type of questions. In particular, we will derive necessary and sufficient conditions on $J$ in order to have solvability of the problems of type (1). On the other hand, we will obtain several near constant uniqueness results. By this we mean that a solution $h$ of one of the previous problems is close to a constant solution $h_*$ then $h=h_*$.

 Notice that if the function $h$ on the right hand side of \eqref{travelscale} is constant, then all the functions $h(\theta,t)$ are the same for all values of $c$. It is then natural to ask why we are studying the different steady state problems for different values of $c$. The reason is that the uniqueness for near constant solutions results that we will prove in the following propositions imply that there are not near constant travelling waves for any value of the velocity $c$. 

Concerning the stationary solutions, the following results holds:
\begin{prop}\label{prop}
The problem \eqref{ode1}, \eqref{odebc} has positive solutions $h\in C\e{3}([0,2\pi])$ if and only if $J>0$. For any $J>0$, there is a constant solution of \eqref{ode1}, \eqref{odebc} given by
\begin{equation}\label{estac1}
h(\theta)=\sqrt{2J}.
\end{equation} 

Moreover, for any $L>0$ with $\frac{1}{L}\leq J\leq L$ there exists $\varepsilon=\varepsilon(L)$ such that the only solution of \eqref{ode1}, \eqref{odebc} satisfying $\norm{h-\sqrt{2J}}_{L\e{\infty}(\T)}<\varepsilon$ is \eqref{estac1}.

The problem \eqref{ode2}, \eqref{odebc} has positive solutions if and only if $J=0$. The corresponding solutions of \eqref{ode2} are given by
\begin{equation}\label{estac2}
h(\theta)=c_1+c_2\sin(\theta-\theta_0),
\end{equation}
where $\abs{c_2}<c_1$, for $c_1>0$, $c_2\in\R$ and $\theta_0\in\R.$

\end{prop}
\begin{nota}
The role of $L$ in this Proposition, as well as the remaining results in this section, is to define the range of the parameter $J$ (or eventually other parameters appearing in the corresponding problems) for which the size of the admissible perturbations, that is measured by $\varepsilon$, is not too small. It will be clear from the proof of Proposition \ref{prop} that a given $J$, $\varepsilon$ cannot be expected to be larger than $C\sqrt{J}$.
\end{nota}
\begin{proof}
We rewrite \eqref{ode1}, using that $h>0$, as:
\begin{equation}\label{ode11}
h\de{}h+h\de{3}h=\frac{J}{h\e{2}}-\frac{1}{2}
\end{equation}

Integrating by parts in the second term on the left hand side of \eqref{ode11} in $[0,2\pi]$ and using that \eqref{odebc} and \eqref{ode1} imply that $\de{j}h(\theta+2\pi)=\de{j}h(\theta)$ for $j=1,2,3,...$; we obtain:

\begin{align}\label{contradic1}
0=\int_{\T}\de{}(\frac{h\e{2}}{2})d\theta-\int_{\T}\de{}\Big(\frac{(\de{}h)\e{2}}{2}\Big)d\theta=\int_{\T}\frac{J}{h\e{2}}d\theta-\int_{\T}\frac{1}{2}d\theta.
\end{align}
where we use the notation $\int_\T\cdot d\theta$ to denote integration in $[0,2\pi]$ with periodic boundary conditions.
Equation \eqref{contradic1} yields a contradiction if $J\leq 0$. Therefore, \eqref{ode1}, \eqref{odebc} have positive solutions only if $J>0$. It is easy to check that $h(\theta)=\sqrt{2J}$ is a solution of \eqref{ode1}, \eqref{odebc}. Then, it only remains to prove that this is the unique solution satisfying $\norm{h-\sqrt{2J}}_{L\e{\infty}(\T)}<\varepsilon$, if $\varepsilon$ is sufficient small.

In order to prove that, we define the functional $F:H\e{3}(\T)\to L\e{2}(\T)$ by means of:
$$F(h)\equiv \frac{h\e{2}}{2}+h\e{3}(\de{}h+\de{3}h).$$

 Then \eqref{ode1}, \eqref{odebc} imply 
 \begin{equation}\label{functional1}
 F(h)=J.
 \end{equation}
 
 The Taylor expansion of $F$ at $c=\sqrt{2J}$ gives,
\begin{equation}\label{taylorode1}
F(h)=F(c)+DF(c)(h-c)+\mathcal{O}(\norm{h-c}_{H\e{3}(\T)}\e{2})\quad\textit{as}\quad \norm{h-c}_{H\e{3}(\T)}\to 0
\end{equation}

where $DF(c):H\e{3}(\T)\to L\e{2}(\T)$ is given by $DF(c)(g)=cg+c\e{3}(\de{}g+\de{3}g)$.


The operator $DF(c)$ is diagonal in the Fourier bases. Indeed if we write $g=\sum_{\ell=-\infty}\e{\infty}a_\ell e\e{i\ell\theta}$ and $DF(c)(g)=\sum_{\ell=-\infty}\e{\infty}\Big(DF(c)(g)\Big)_\ell e\e{i\ell\theta}$ with

\begin{displaymath}
\Big(DF(c)(g)\Big)_\ell=(c+c\e{3}i(\ell-\ell\e{3}))a_\ell
\end{displaymath}
Therefore, we have that $F(c)=J$ and $DF(c)$ is invertible, then the uniqueness of the solution $h$ of \eqref{ode1}, \eqref{odebc} stated in the Proposition follows from an Inverse Function Theorem argument. 
Indeed using \eqref{functional1}, \eqref{taylorode1} we have,
\begin{displaymath}
\norm{DF(c)(h-c)}_{L\e{2}(\T)}\leq C\norm{h-c}\e{2}_{H\e{3}(\T)}
\end{displaymath}

then, the invertibility of $DF(c)$ implies:

\begin{displaymath}
\norm{h-c}_{H\e{3}(\T)}\leq C\norm{h-c}\e{2}_{H\e{3}(\T)}
\end{displaymath}
when $h=c$ if $\norm{h-c}_{H\e{3}(\T)}\leq\varepsilon$ with $\varepsilon$ sufficient small.
On the other hand, standard ODE arguments imply that $\norm{h-c}_{H\e{3}(\T)}\leq C\norm{h-c}_{L\e{\infty}(\T)}$ if $\norm{h-c}_{L\e{\infty}(\T)}\leq \frac{c}{2}$, whence  the uniqueness result stated in the Proposition follows.

Suppose that $h$ is a positive solution of the problem \eqref{ode2}, \eqref{odebc}. We rewrite \eqref{ode2} as
\begin{equation}\label{ode21}
h\de{}h+h\de{3}h=\frac{J}{h\e{2}}
\end{equation}

Using integration by parts in the second term of the left hand side of the equation \eqref{ode21} we deduce:
\begin{displaymath}
0=\int_{\T}\de{}(\frac{h\e{2}}{2})d\theta-\int_{\T}\de{}\Big(\frac{(\de{}h)\e{2}}{2}\Big)d\theta=\int_{\T}\frac{J}{h\e{2}}d\theta.
\end{displaymath} 

Hence, $J=0$. Therefore, \eqref{ode2} reduces to $\de{}h+\de{3}h=0$ whose most general solution satisfying $h>0$, is given by \eqref{estac2}.

\end{proof}

We recall that the interface has been parametrized in polar coordinates by $r=1+\varepsilon h(\theta,t)$. In addition to this, we make the changes of coordinates to a rotating coordinate system (cf. \eqref{rescale}, \eqref{reescalefuera} and \eqref{reescalegrande}). We will refer to the coordinate systems as original and rotating respectively.

Geometrically the interface described by the solution \eqref{estac1} is in both coordinate systems a circle centred at the origin. Notice that origin is also the center of the internal and external cylinders enclosing the flow.

The solutions \eqref{estac2}, if $c_2\neq 0$, describe in the rotating coordinate system circles whose center is different from the origin. In the original coordinate system, these solutions are circles whose center is separated from the origin and rotates with constant angular speed around the origin. If $c_2=0$, the solutions \eqref{estac2} are circles centred at the origin in both coordinate systems.

\begin{prop}\label{proptravel}
\hspace{2cm}
\begin{enumerate}
\item[(a)] For all $h_*>0$, $h_*\in\R$ and any $c\in\R$ the constant function $h=h_*$ is solution of the problem \eqref{travel1}, \eqref{odebc} with $J=-ch_*+\frac{h_*\e{2}}{2}$. Moreover, for any $L>0$ there exists $\varepsilon=\varepsilon(L)>0$ such that for each $c\in[-L,L]$, $J\in [-\frac{c\e{2}}{2},L]$, $\frac{1}{L}\leq h_*\leq L$ and any $h$ solution of \eqref{travel1}, \eqref{odebc} satisfying $\norm{h-h_*}_{L\e{\infty}}\leq\varepsilon$ we have that $h=\tilde{h}\in\mathcal{S}_{\varepsilon,h_*,c}$  where $\mathcal{S}_{\varepsilon,h_*,c}=\{x\in[h
_*-\varepsilon,h_*+\varepsilon]:-cx+\frac{x\e{2}}{2}=-ch_*+\frac{h_*\e{2}}{2}\}$.
\item[(b)] For all $h_*>0$, $h_*\in\R$ and any $c\in\R$, the constant function $h=h_*$ is solution of the problem \eqref{travel2}, \eqref{odebc} with $J=-ch_*$. Moreover, for any $L>0$ there exists $\varepsilon=\varepsilon(L)>0$ such that for each $J\in[-L,L]$, $\frac{1}{L}\leq c \leq L$ and any $h$ solution of the problem \eqref{travel2}, \eqref{odebc} satisfying $\norm{h-h_*}_{L\e{\infty}}\leq\varepsilon$ we have that $h=h_*$.
\end{enumerate}
\end{prop}
\begin{nota} Concerning the size of the admissible perturbations $\varepsilon$, it will be seen in the proof that we need to assume $\varepsilon<h_*$. In particular, $\varepsilon$ must be small if $h_*$ is small. Furthermore, notice that  in the case (b) we cannot expect $\varepsilon$ uniformly away from $0$ if $c\to 0$  because, as we have seen in Proposition \ref{prop}, there exists solutions to the problem different from constant if $c=0$.
\end{nota}
\begin{nota}
The set $\mathcal{S}_{\varepsilon,h_*,c}$ is not empty because $h_*\in\mathcal{S}_{\varepsilon,h_*,c}$. The set $\mathcal{S}_{\varepsilon,h_*,c}$ contains one or two elements depending of the values of $\varepsilon$, $h_*$ and $c$. Notice that for equations \eqref{travel1}, \eqref{odebc} we do not have near constant uniqueness results for each value of $J$ in general. However, Proposition \ref{proptravel}(a) implies that for any given value of $J$ there are at most two solutions of \eqref{travel1}, \eqref{odebc} close to the value $h_*$ and both of them are constant. This is a consequence of the nonmonotonicity of the function that give $J$ as a function of $h_*$ (cf. Figure \ref{graficaJ}) that allows to obtain two constant solutions of \eqref{travel1}, \eqref{odebc} that are arbitrarily close in the $L\e{\infty}$-norm.
\end{nota}
\begin{figure}[htb]\label{graficaJ}
\includegraphics[width=100mm]{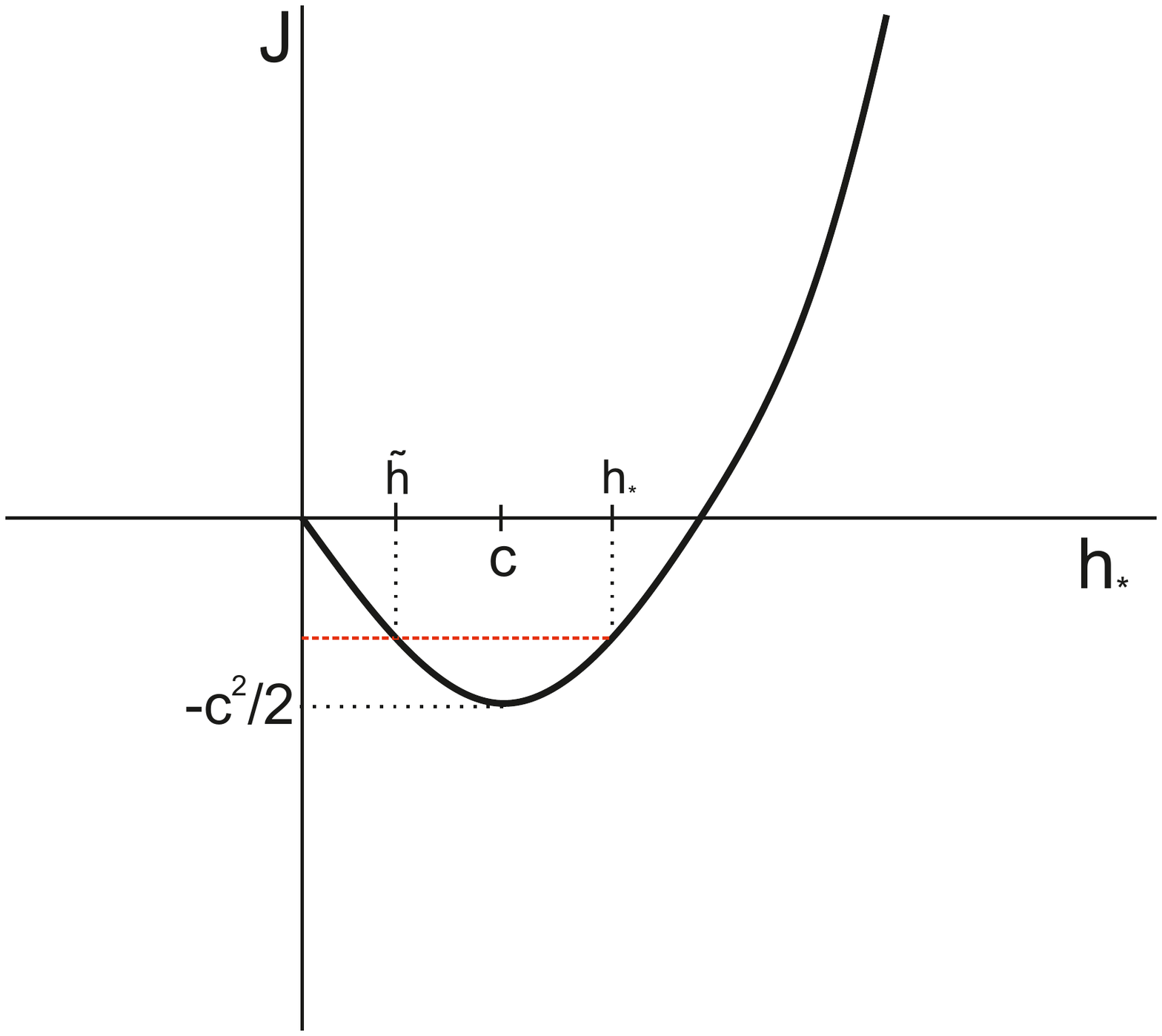}
\caption{Graphic of $J(h_*)=-ch_*+\frac{h_*\e{2}}{2}$}
\end{figure}

\begin{proof}
We prove first (a).
It is trivial to see that $h=h_*$ is solution of the problem \eqref{travel1}, \eqref{odebc} with $J=-ch_*+\frac{h_*\e{2}}{2}$. Therefore, we only have to prove that $\norm{h-h_*}_{L\e{\infty}}(\T)\leq\varepsilon$ for $\varepsilon>0$ small implies $h$ is constant. 

Standard ODE methods imply
\begin{equation}\label{standardestimate}
\norm{h-h_*}_{H\e{3}(\T)}\leq C\norm{h-h_*}_{L\e{\infty}(\T)}
\end{equation}
if $\varepsilon\leq \frac{h_*}{2}$ where in the rest of the proof $C$ is a generic constant depending on $L$.

We define the functional $F:H\e{3}(\T)\to L\e{2}(\T)$ as

\begin{equation}\label{functional2}
F(h)\equiv -ch+\frac{h\e{2}}{2}+h\e{3}(\de{}h+\de{3}h).
\end{equation}
Then, we can reformulate \eqref{travel1},\eqref{odebc} as
\begin{equation}\label{functional2eq}
F(h)=J.
\end{equation}
We define $V_0\equiv Span\{cos(\theta),sin(\theta)\}\subset L\e{2}(\T)$ and $V_1$ as the orthogonal of $V_0\oplus Span\{1\}$ in $L\e{2}(\T)$. We denote as $P_0$ and $P_1$ the orthogonal projection of $L\e{2}(\T)$ in $V_0$ and $V_1$ respectively. 

Given $h\in H\e{3}(\T)$ we can decompose it as 
\begin{equation}\label{descomposicion}
h-h*=a_0+w+\zeta
\end{equation}
 with $a_0\in\R$, $w\in V_0\cap H\e{3}(\T)$ and $\zeta\in V_1\cap H\e{3}(\T)$. Notice that (cf. \eqref{standardestimate}) 

\begin{equation}\label{estepsilon}
\abs{a_0}\leq C\varepsilon,\quad\norm{w}_{H\e{3}(\T)}\leq C\varepsilon\quad\textit{and}\quad\norm{\zeta}_{H\e{3}(\T)}\leq C\varepsilon.
\end{equation}  

We expand $F$ using Taylor formula as
\begin{equation}\label{taylortravel1}
F(h)= F(h_*+a_0)+DF(h_*+a_0)(h-h_*-a_0)+\frac{1}{2}D\e{2}F(h_*+a_0)(h-h_*-a_0,h-h_*-a_0)+\mathcal{O}(\norm{h-h_*-a_0}_{H\e{3}(\T)}\e{3})
\end{equation}

where $DF(h_*+a_0):H\e{3}(\T)\to L\e{2}(\T)$ is given by
\begin{equation}\label{primeraderivtravel1}
DF(h_*+a_0)(g)=(h_*+a_0-c)g+(h_*+a_0)\e{3}(\de{}g+\de{3}g)
\end{equation}

and $D\e{2}F(h_{*}):H\e{3}(\T)\times H\e{3}(\T)\to L\e{2}(\T)$ is determined by the quadratic form

\begin{equation}\label{segundaderivtravel2}
\frac{1}{2}D\e{2}F(h_{*}+a_0)(g,g)= \frac{g\e{2}}{2}+3(h_{*}+a_0)\e{2}g(\de{}g+\de{3}g)
\end{equation}

Now, we define $\tilde{h}=h_*+a_0$ to discard the notation. Using that $F(\tilde{h})=J+(h_*-c)a_0+\frac{a_0\e{2}}{2}$, as well as \eqref{functional2eq} and \eqref{taylortravel1}
\begin{equation}\label{travel1taylor0}
(h_*-c)a_0+\frac{a_0\e{2}}{2}+DF(\tilde{h})(h-\tilde{h})+\frac{1}{2}D\e{2}F(\tilde{h})(h-\tilde{h},h-\tilde{h})+\mathcal{O}(\norm{h-\tilde{h}}_{H\e{3}(\T)}\e{3})=0
\end{equation}

Applying the projection $P_1$ to \eqref{travel1taylor0} and using \eqref{primeraderivtravel1},  \eqref{segundaderivtravel2} and \eqref{descomposicion},
\begin{equation}\label{projection1travel1}
(P_1DF(\tilde{h})P_1)(\zeta)+P_1(\frac{1}{2}(w\e{2}))+\mathcal{O}(\norm{w}_{H\e{3}(\T)}\norm{\zeta}_{H\e{3}(\T)})+\mathcal{O}(\norm{\zeta}_{H\e{3}(\T)}\e{2})+\mathcal{O}(\norm{w}_{H\e{3}(\T)}\e{3})=0
\end{equation}
where $P_1DF(\tilde{h})P_1:V_1\cap H\e{3}(\T)\to V_1$ is given by 
\begin{equation}\label{p1derivada}
(P_1DF(\tilde{h})P_1)(\zeta)=(\tilde{h}-c)\zeta+(\tilde{h})\e{3}(\de{}\zeta+\de{3}\zeta)
\end{equation}
The operator \eqref{p1derivada} can be represented
using Fourier as,
\begin{displaymath}
\Big((P_1DF(\tilde{h})P_1)(\zeta)\Big)_{\ell}=((\tilde{h}-c)+i(\tilde{h})\e{3}(\ell-\ell\e{3}))a_\ell\quad\textit{for}\quad \ell\neq 0,-1,1.
\end{displaymath}

where $\zeta(\theta)=\sum_{\ell=-\infty, \ell\neq 0,1,-1}\e{\infty}a_\ell e\e{i\ell\theta}$ and $(P_1DF(\tilde{h})P_1)(\zeta)=\sum_{\ell=-\infty, \ell\neq 0,1,-1}\e{\infty}\Big((P_1DF(\tilde{h})P_1)(\zeta)\Big)_\ell e\e{i\ell\theta}$.

Then, using the fact that $\tilde{h}\geq \frac{h_*}{2}$ if $\varepsilon$ is sufficiently small (cf. \eqref{estepsilon}), the operator $\Big((P_1DF(\tilde{h})P_1)(\zeta)\Big)\e{-1}:V_1\to V_1\cap H\e{3}(\T)$ exists and it is bounded.

Thus, 
\begin{equation}\label{zeta}
\zeta=\zeta_1+\zeta_2
\end{equation}
where 
\begin{equation}\label{zeta1}
\zeta_1=-\Big((P_1DF(\tilde{h})P_1)(\zeta)\Big)\e{-1}(P_1(\frac{1}{2}(w\e{2})))
\end{equation}
and
\begin{equation}\label{zeta2}
\zeta_2=-\Big((P_1DF(\tilde{h})P_1)(\zeta)\Big)\e{-1}(Z)\quad\textit{with}\quad \norm{Z}_{L\e{2}(\T)}\leq C(\norm{w}_{H\e{3}(\T)}\norm{\zeta}_{H\e{3}(\T)}+\norm{\zeta}_{H\e{3}(\T)}\e{2}+\norm{w}_{H\e{3}(\T)}\e{3})
\end{equation}

Identity \eqref{zeta1} implies that $\norm{\zeta_1}_{H\e{3}(\T)}\leq C\norm{w}\e{2}_{H\e{3}(\T)}$. Combining this estimate with \eqref{zeta} and \eqref{zeta2}, we obtain 
\begin{equation}\label{estimatzeta2}
\norm{\zeta_2}_{H\e{3}(\T)}\leq C(\norm{w}\e{3}_{H\e{3}(\T)}+\norm{w}_{H\e{3}(\T)}\norm{\zeta_2}_{H\e{3}(\T)}+\norm{\zeta_2}_{H\e{3}(\T)}\e{2}).
\end{equation} 

Using \eqref{estepsilon}, \eqref{zeta2} we obtain that $\norm{\zeta_2}_{H\e{3}(\T)}\leq C\varepsilon$ if $\varepsilon$ sufficiently small. Then, \eqref{estimatzeta2} yields,
\begin{equation}\label{ordenzeta2}
\norm{\zeta_2}_{H\e{3}(\T)}\leq C\norm{w}\e{3}_{H\e{3}(\T)}
\end{equation}

In order to obtain a more explicit formula for $\zeta_1$, we write $w=a_{-1} e\e{-i\theta}+a_1 e\e{i\theta}$. Then $$P_1(\frac{w\e{2}}{2})=\frac{1}{2}(a_{1}\e{2}e\e{2i\theta}+a_{-1}\e{2}e\e{-2i\theta})$$
and using \eqref{zeta1} we derive

\begin{equation}\label{travelorto}
\zeta_1=-\frac{1}{2(\tilde{h}-c)-12\tilde{h}\e{3}i}a_{1}^2 e^{2i\theta}-\frac{1}{2(\tilde{h}-c)+12\tilde{h}\e{3}i}a_{-1}^2 e^{-2i\theta}.
\end{equation}

Now applying the projection $P_0$ to \eqref{travel1taylor0} and using \eqref{primeraderivtravel1}, \eqref{segundaderivtravel2}, \eqref{estimatzeta2} and \eqref{travelorto}, as well as the fact that $P_0(\de{}+\de{3})=0$ and $P_0(1)=0$, we arrive at:
\begin{equation}\label{p0travel1}
(\tilde{h}-c)w+P_0(w\zeta_1)+P_0(3\tilde{h}\e{2}w(\de{}\zeta_1+\de{3}\zeta_1))=\mathcal{O}(\norm{w}\e{4}_{H\e{3}(\T)})
\end{equation}
Using
\begin{align*}
&P_0(w\zeta_1)=-\frac{1}{2(\tilde{h}-c)-12\tilde{h}\e{3}i}a_{-1}a_{1}^2 e^{i\theta}-\frac{1}{2(\tilde{h}-c)+12\tilde{h}\e{3}i}a_1a_{-1}^2 e^{-i\theta};\\
&P_0(3(\tilde{h})\e{2}w(\de{}\zeta_1+\de{3}\zeta_1))=\\
&=18\tilde{h}\e{2}i(\frac{1}{2(\tilde{h}-c)-12\tilde{h}\e{3}i}a_{-1}a_{1}^2 e^{i\theta}-\frac{1}{2(\tilde{h}-c)+12\tilde{h}\e{3}i}a_1a_{-1}^2 e^{-i\theta}),
\end{align*}
we can rewrite \eqref{p0travel1} as
\begin{equation}\label{sistma1}
\left\{ \begin{array}{ll}
\abs{(\tilde{h}-c)a_1+\alpha a_{-1}a_1\e{2}}=\mathcal{O}(\norm{w}\e{4}_{H\e{3}(\T)})\\
\abs{(\tilde{h}-c)a_{-1}+\bar{\alpha} a_{1}a_{-1}\e{2}}=\mathcal{O}(\norm{w}\e{4}_{H\e{3}(\T)})
\end{array} \right.
\end{equation}

with $\alpha=\alpha_1+i\alpha_2$ where 
\begin{equation}\label{alpha1alpha2}
\alpha_1=\frac{-(\tilde{h}-c)-108\tilde{h}\e{5}}{2(\tilde{h}-c)\e{2}+72\tilde{h}\e{6}},\quad\alpha_2=\frac{9(\tilde{h}-c)\tilde{h}\e{2}-3\tilde{h}\e{3}}{(\tilde{h}-c)\e{2}+36\tilde{h}\e{6}}.
\end{equation}

Using that $a_1=\overline{a_{-1}}$ (cf. \eqref{espacioswzeta}), we can reduce \eqref{sistma1} to:
\begin{displaymath}
\abs{(\tilde{h}-c+\alpha_1\abs{a_1}\e{2})a_1+i\alpha_2 \abs{a_1}\e{2}a_1}=\mathcal{O}(\abs{a_1}\e{4}).
\end{displaymath}
Suppose now that $\abs{a_1}\neq 0$, therefore:
\begin{displaymath}
\abs{(\tilde{h}-c+\alpha_1\abs{a_1}\e{2})+i\alpha_2 \abs{a_1}\e{2}}\leq C\abs{a_1}\e{3}
\end{displaymath}

Thus, using that due to \eqref{estepsilon}, we have $\abs{a_1}\leq C\varepsilon$. Then,
\begin{align*}
&(\frac{1}{2}\abs{\tilde{h}-c}\e{2}+\alpha_2\e{2})\abs{a_1}\e{4}\leq\frac{1}{2}\abs{\tilde{h}-c}\e{2}+\alpha_2\e{2}\abs{a_1}\e{4}\leq\abs{\tilde{h}-c+\alpha_1\abs{a_1}\e{2}a_1}\e{2}+\alpha_2\e{2}\abs{a_1}\e{4}\\
&= \abs{(\tilde{h}-c+\alpha_1\abs{a_1}\e{2})+i\alpha_2 \abs{a_1}\e{2}}\e{2}\leq C\abs{a
_1}\e{6}
\end{align*}

The definition \eqref{alpha1alpha2} implies that $\frac{1}{2}\abs{\tilde{h}-c}\e{2}+\alpha_2\e{2}\geq C_0(L)>0$ whence $\abs{a_1}\geq C_1(L)$. This contradicts the fact that $\abs{a_1}\leq C\varepsilon$ if $\varepsilon$ is sufficiently small. Therefore, we have that $\abs{a_1}=0$. Thus, $w=\zeta_1=0$ and $h=\tilde{h}$. Therefore, $h$ is a constant $\tilde{h}$ given by $\tilde{h}=h_*+a_0$ where $\tilde{h}$ solves $J=-c\tilde{h}+\frac{1}{2}\tilde{h}\e{2}$.

We now consider the case (b). It directly follows that $h=h_*$ is solution to problem \eqref{travel2}, \eqref{odebc} with $J=-ch_*$. We define $F:H\e{3}(\T)\to L\e{2}(\T)$ by means of
\begin{equation}\label{functional3}
F(h)\equiv -ch+h\e{3}(\de{}h+\de{3}h)\quad\textit{for}\quad c\neq 0
\end{equation}
We have that 
\begin{displaymath}
DF(h_*)(h)=-ch+h_*\e{3}(\de{}h+\de{3}h).
\end{displaymath}
For $c\neq 0$, $DF(h_*)$ is invertible and $F(h_*)=J$, therefore the result follows arguing as in the proof of Proposition \ref{prop} in the case of $J>0$.
\end{proof}

We now consider the problems of Type (2) which were introduced at the beginning of this Section. Since $J$ is not given a priori it is convenient to reformulate the problems in a way that $J$ does not appear.

Differentiating equation \eqref{ode1} we have

\begin{equation}\label{derivode1}
\de{}(\frac{h\e{2}}{2})+\de{}(h^3(\de{}{h}+\partial_\theta^3 h))=0,
\end{equation}

and for \eqref{ode2},
\begin{equation}\label{derivode2}
\de{}(h^3(\de{}{h}+\partial_\theta^3 h))=0,
\end{equation}
In both cases we must have in addition 
\begin{equation}\label{derivbc}
h(\theta+2\pi)=h(\theta)\quad\textit{for}\quad\theta\in\R.
\end{equation}

Equations for traveling waves solutions of \eqref{travel1} and \eqref{travel2} can be written as

\begin{equation}\label{derivtravel1}
-c\de{}h+\de{}(\frac{h\e{2}}{2})+\de{}(h^3(\de{}{h}+\partial_\theta^3 h))=0
\end{equation}
and
\begin{equation}\label{derivtravel2}
-c\de{}h+\de{}(h^3(\de{}{h}+\partial_\theta^3 h))=0.
\end{equation}
In both cases $h$ satisfies \eqref{derivbc}.
\begin{prop}\label{prop3}\hspace{2cm}

(a)Let be $h_*>0$. Suppose that $h$ is a solution of one of the problems $[$\eqref{derivode1},\eqref{derivbc}$]$. For any $L>0$ there exists $\varepsilon=\varepsilon(L)>0$ such that $\frac{1}{L}\leq h_*\leq L$ and $h$ satisfies $\norm{h-h_*}_{L\e{\infty}(\T)}\leq\varepsilon$ and $\int_{\T}h(\theta)d\theta=2\pi h_*$, then $h=h_*$.

(b)There exists a two-parameter family of positive solutions $h\in H\e{4}(\T)$ of \eqref{derivode2}, \eqref{derivbc} with $\int_{\T}h(\theta)d\theta=2\pi h_*$ that is given by:
\begin{displaymath}
h(\theta)=c_1+c_2\sin(\theta-\theta_0),
\end{displaymath}
where $\abs{c_2}<c_1$, for $c_1>0$, $c_2\in\R$ and $\theta_0\in\R.$

(c)Let be $h_*>0$. Suppose that $h$ is a solution of one of the problems $[$\eqref{derivtravel1},\eqref{derivbc}$]$ and $[$\eqref{derivtravel2},\eqref{derivbc}$]$. For any $L>0$ there exists $\varepsilon=\varepsilon(L)>0$ such that $\abs{c}\in[-L,L]$, $\frac{1}{L}\leq h_*\leq L$ and $h$ satisfies $\norm{h-h_*}_{L\e{\infty}(\T)}\leq\varepsilon$ and $\int_{\T}h(\theta)d\theta=2\pi h_*$, then $h=h_*$. 
\end{prop}
\begin{proof}
The proof of (b) follows from the fact that under the assumptions in the Proposition $h$ solves \eqref{ode2}, \eqref{odebc} for some $J\in\R$. Using then Proposition \ref{prop} we obtain the desired result.
 
Notice that in both cases (a) and (c), the assumptions in the Proposition imply
\begin{equation}
\norm{h-h_*}_{H\e{4}(\T)}\leq C\norm{h-h_*}_{L\e{\infty}(\T)}\leq C\varepsilon
\end{equation}
if $\varepsilon$ is sufficiently small (depending on $L$) and $C=C(L)$.

Integrating \eqref{derivode1}, \eqref{derivtravel1} and \eqref{derivtravel2}, we obtain that $h$ satisfies third order differential equations which can be reformulated as \eqref{functional1}, \eqref{functional2} and \eqref{functional3} by means of suitable functionals $F:H\e{3}(\T)\to L\e{2}(\T)$. For instance, in the case of equation \eqref{derivtravel1} we define: 

\begin{equation}\label{derivfunctional2}
F(h)\equiv -ch+\frac{h\e{2}}{2}+h\e{3}(\de{}h+\de{3}h).
\end{equation}

Then, we can reformulate \eqref{derivtravel1},\eqref{derivbc} as
\begin{equation}\label{derivfunctional2eq}
F(h)=J.
\end{equation}

We expand $F$ using Taylor formula as
\begin{equation}\label{derivtaylortravel1}
F(h)= F(h_*)+DF(h_*)(h-h_*)+\frac{1}{2}D\e{2}F(h_*)(h-h_*,h-h_*)+\mathcal{O}(\norm{h-h_*}_{H\e{3}(\T)}\e{3})
\end{equation}
The same expansion can be derived for the different functionals $F$ associated to each other problems. We define $F(h_*)=J_*$ and using \eqref{derivfunctional2eq} and \eqref{derivtaylortravel1}
\begin{equation}\label{derivtravel1taylor0}
J_*-J+DF(h_*)(h-h_*)+\frac{1}{2}D\e{2}F(h_*)(h-h_*,h-h_*)+\mathcal{O}(\norm{h-h_*}_{H\e{3}(\T)}\e{3})=0.
\end{equation}

We now define $V_0\equiv Span\{cos(\theta),sin(\theta)\}\subset L\e{2}(\T)$ and we denote by $V_1$ as the orthogonal of $V_0\oplus Span\{1\}$ in $L\e{2}(\T)$. We then write $P_0$ and $P_1$ the orthogonal projection of $L\e{2}(\T)$ in $V_0$ and $V_1$ respectively. We have that we can write $h=h_*+w+\zeta$ with $w\in V_0$ and $\zeta\in V_1$. Notice that we do not need to add a constant Fourier mode $a_0$ because $\int_{\T}h(\theta)d\theta=2\pi h_*$.

 Therefore, we can apply $P_0$ and $P_1$ to \eqref{derivtravel1taylor0}, and using that $P_0(J_*-J)=P_1(J_*-J)=0$, we can continue with the same arguments as in the proofs of Propositions \ref{prop} and \ref{proptravel}.

\end{proof}

\begin{nota}\label{nota} 
We will refer, from now on, to all the solutions described in Proposition \ref{prop} and Proposition \ref{prop3}(a) (including the non-constant solutions \eqref{estac1}, \eqref{estac2}) as \textbf{circular steady states}. We will refer to the solutions described in Proposition \ref{proptravel} and in Proposition \ref{prop3}(b) as \textbf{circular travelling waves}. The term circular arises from the fact that the interface $\{r=1+\varepsilon h\}$ is a circle for the above mentioned solutions. Notice that Propositions \ref{prop}, \ref{proptravel} and \ref{prop3} do not rule out the possibility of having non-constant solutions of $[$\eqref{ode1}, \eqref{odebc}$]$ with $J>0$ $[$\eqref{travel1}, \eqref{odebc}$]$,$[$\eqref{travel2}, \eqref{odebc}$]$,$[$\eqref{derivode1}, \eqref{derivbc}$]$,$[$\eqref{derivtravel1}, \eqref{derivbc}$]$ and $[$\eqref{derivtravel2}, \eqref{derivbc}$]$. Such solutions, that geometrically do not describe circular interfaces, will not be considered in this paper.
\end{nota}

\section{Stability of the circular steady states}\label{S5}

In this Section, we study the stability of the circular steady states and travelling waves described in Propositions \ref{prop3} (see also Remark \ref{nota}). We consider separately the cases in which $\gamma\approx\frac{b}{\varepsilon\e{2}}$ with $b>0$ and $\gamma\gg\frac{1}{\varepsilon\e{2}}$.

\subsection{The case $\gamma\approx\frac{b}{\varepsilon\e{2}}$}\label{S5ss1}

As we have seen in Subsection \ref{S3s1} the evolution of the interface can be approximated in this case by means of the equation \eqref{equevol}.
We will next prove the following global well posedness result:
\begin{thm}\label{thglobal}
Let $c>0$. There exists $\varepsilon>0$ (depending on $c$) such that, for any $h_0\in H\e{4}(\T)$ satisfying $\norm{h_0-c}_{H\e{4}(\T)}<\varepsilon$ with $\frac{1}{2\pi}\int_{\T}h_0=c$, there exists a unique solution $h\in C([0,\infty); H\e{4}(\T))\cap C\e{1}((0,\infty);H\e{4}(\T))$ of \eqref{equevol}, where $h(\cdot,0)=h_0(\cdot)$.

Moreover, we have
\begin{equation}\label{cotathglobal}
\norm{h(\cdot,t)-c}_{H\e{4}(\T)}\leq \bar{C}\varepsilon\quad\textit{for all}\quad t\geq 0
\end{equation}
where $\bar{C}$ depend only on $c$. 
\end{thm}

\begin{nota}
The dependence of the right hand side of \eqref{cotathglobal} on $\varepsilon$ is far from optimal. Better estimates concerning the behaviour of $h$ for long times will be prove in Theorem \ref{thestab}.
\end{nota}

In order to prove Theorem \ref{thglobal} it is convenient to reformulate \eqref{equevol} in a rotating coordinate system at velocity $c$. Moreover, we also linearize around the constant solution $h=c$.

More precisely, we define
\begin{equation}\label{defvfuncionh}
v(\varphi,t)=h(\theta,t)-c\quad\textit{with}\quad\varphi=\theta-ct.
\end{equation} 

Then using \eqref{equevol} we obtain that $v$ solves
\begin{equation}\label{evoleq}
\frac{dv}{dt}=L(v)+R(v)
\end{equation}
where the linear operator $L:\dhil{4}\to L\e{2}(\T)$ is:
\begin{equation}\label{defl}
L(v)=-c^3(\partial_\varphi^2 v+\partial_\varphi^4 v)\quad\textit{for each}\quad v\in\dhil{4}
\end{equation}
and the non-linear operator $R:\dhil{4}\to L\e{2}(\T)$ is:
\begin{equation}\label{defr}
R(v)=-\dfi{}(\frac{v^2}{2}) -\dfi{}((v^3+3c^2v+3cv^2)(\partial_\varphi v+\partial_\varphi^3 v))
\end{equation}

It immediately follows that $L$ is a well defined operator from $\dhil{4}$ to $L\e{2}(\T)$. The fact that the operator $R$ is well defined from $\dhil{4}$ to $L\e{2}(\T)$ is just a consequence from the embedding $\dhil{s}\subset L\e{\infty}(\T)$ for any $s\geq 1$.

We recall that $V_0\equiv Span\{cos(\theta),sin(\theta)\}\subset L\e{2}(\T)$ and $V_1$ as the orthogonal of $V_0\oplus Span\{1\}$ in $L\e{2}(\T)$. 
We define the following subspaces of $\dhil{4}$,
\begin{equation}\label{e0e1}
\mathcal{E}_0=V_{0}\cap\dhil{4}\quad\textit{and}\quad\mathcal{E}_1=V_{1}\cap\dhil{4}
\end{equation}
and we denote as $P_0$ and $P_1$ the orthogonal projections of $L\e{2}(\T)$ into $V_0$ and $V_1$ respectively. We have that $\dhil{4}=\mathcal{E}_0\oplus\mathcal{E}_1$. Notice that, using Fourier, it readily follows that $P_0(\dhil{4})\subset \mathcal{E}_0$, and $P_1(\dhil{4})\subset\mathcal{E}_1$. In all this subsection, we only apply the operators $P_0$ and $P_1$ to functions of $\dhil{4}$.

Given $v\in\dhil{4}$, we can then write $v=v_0+v_1$ with $v_0=P_0(v)\in\mathcal{E}_0$ and $v_1=P_1(v)\in \mathcal{E}_{1}$. 

We defined the quadratic operator, $Q:\mathcal{E}_0\times\mathcal{E}_0\to\mathcal{E}_1$ as
\begin{equation}\label{defq}
Q(v_0)=\frac{i}{24c^3}a_{-1}^2 e^{-2i\varphi}-\frac{i}{24c^3}a_{1}^2 e^{2i\varphi}
\end{equation}
for each $v_0=a_{-1}e\e{-i\varphi}+a_1e\e{i\varphi}\in\mathcal{E}_0$, where $a_1=\overline{a_{-1}}$. We notice for further reference that 
\begin{equation}\label{ldeq}
L(Q(v_0))=\dfi{}P_1(\frac{v_0\e{2}}{2})
\end{equation}
which follows from \eqref{defl} as well as the fact that 
\begin{displaymath}
c\e{3}(\dfi{2}Q(v_0)+\dfi{4}Q(v_0))+\dfi{}P_1(\frac{v_0\e{2}}{2})=0.
\end{displaymath}

As a first step to prove Theorem \ref{thestab} we need a local existence result for \eqref{evoleq}. In order to avoid breaking the continuity of the arguments, this result will be postponed to the Appendix \ref{appendix} (cf. Proposition \ref{wellposedness}).

In the next two Lemmas we decompose $v$ as the sum of functions in $\mathcal{E}_0$ and $\mathcal{E}_1$ and rewrite \eqref{evoleq} in a convenient way to derive global a priori estimates for their solutions.

\begin{lem}\label{F0F1}
Suppose that $v\in C([0,T];\dhil{4})\cap C\e{1}((0,T];\dhil{4})$, with $T>0$, is a solution of the problem \eqref{evoleq}-\eqref{defr}.  Moreover, let us assume also that $\norm{v(\cdot,t)}_{\dhil{4}}\leq 1$ for each $t\in[0,T]$. Let be $v_0(\cdot,t)=P_0(v(\cdot,t))$ and $v_1(\cdot,t)=P_1(v_1(\cdot,t))$, then we have
\begin{align}\label{odeu0}
&\frac{dv_0}{dt}=-P_0\Big(\dfi{}\Big(v_0Q(v_0)\Big)\Big)-3c\e{2}P_0\Big(\dfi{}\Big(v_0(\dfi{}Q(v_0)+\dfi{3}Q(v_0))\Big)\Big)\\\nonumber
&+F_0,
\end{align}
\begin{equation}\label{uh-q}
\frac{d}{dt}(v_1-Q(v_0))-L(v_1-Q(v_0))=F_1
\end{equation} 
with $Q(v_0)$ as in \eqref{defq}, where \begin{equation}\label{F0}
\norm{F_0}_{L\e{2}(\T)}\leq C(\norm{v_0}_{\dhil{4}}\e{4}+\norm{v_1-Q(v_0)}\e{2}_{\dhil{4}}+\norm{v_0}_{\dhil{4}}\norm{v_1-Q(v_0)}_{\dhil{4}})
\end{equation} and 
\begin{equation}\label{F1}
\norm{F_1}_{L\e{2}(\T)}\leq C(\norm{v_0}\e{3}_{\dhil{4}}+\norm{v_1-Q(v_0)}\e{2}_{\dhil{4}}+C\norm{v_0}_{\dhil{4}}\norm{v_1-Q(v_0)}_{\dhil{4}}).
\end{equation}
The constant $C$ depends only on $c$ but is independent on $v$. 
\end{lem}

\begin{proof}
We will use repeatedly the following estimate
\begin{equation}\label{vhcota}
\norm{v_1}_{\dhil{4}}\leq\norm{v_1-Q(v_0)}_{\dhil{4}}+\norm{Q(v_0)}_{\dhil{4}}\leq \norm{v_1-Q(v_0)}_{\dhil{4}}+C\norm{v_0}_{\dhil{4}}\e{2}.
\end{equation}

Notice also that since $\norm{v}_{\dhil{4}}\leq 1$ then $\norm{v_0}_{\dhil{4}}+\norm{v_1}_{\dhil{4}}\leq C$.

Applying $P_0$ to \eqref{evoleq}, we obtain 
\begin{equation}\label{p0eq}
\frac{dv_0}{dt}+P_0(\dfi{}(v_0v_1))+3c\e{2}P_0(\dfi{}(v_0(\dfi{}v_1+\dfi{3}v_1)))=\tilde{F}_0
\end{equation}

 where $\tilde{F}_0$ is given by 
\begin{displaymath}
\tilde{F}_0=-P_0(\dfi{}(\frac{v_0\e{2}}{2}))-P_0(\dfi{}(\frac{v_1\e{2}}{2}))-P_0(\dfi{}((v\e{3}+3c\e{2}v_1+3cv\e{2})(\dfi{}v_1+\dfi{3}v_1)))
\end{displaymath} 
using that $\dfi{}v_0+\dfi{3}v_0=0$. 

Notice that  $P_0(\dfi{}(\frac{v_0\e{2}}{2}))=0$. On the other hand, $$\norm{P_0(\dfi{}(\frac{v_1\e{2}}{2}))}_{L\e{2}(\T)}\leq C\norm{v_1}\e{2}_{\dhil{4}}\leq C(\norm{v_0}\e{4}_{\dhil{4}}+\norm{v_1-Q(v_0)}_{\dhil{4}}\e{2})$$ and using \eqref{vhcota} we obtain $$\norm{P_0(\dfi{}((v\e{3}+3c\e{2}v_1+3cv\e{2})(\dfi{}v_1+\dfi{3}v_1)))}_{L\e{2}(\T)}\leq C\norm{v}_{\dhil{4}}\e{2}\norm{v_1}_{\dhil{4}}\leq C(\norm{v_0}\e{4}_{\dhil{4}}+\norm{v_1-Q(v_0)}_{\dhil{4}}\e{2}).$$

Therefore,
\begin{equation}\label{F0tilde}
\norm{\tilde{F}_0}_{L\e{2}(\T)}\leq C(\norm{v_0}\e{4}_{\dhil{4}}+\norm{v_1-Q(v_0)}_{\dhil{4}}\e{2})
\end{equation}

Writing in the second and third term in the left hand side of \eqref{p0eq}, $v_1=Q(v_0)+(v_1-Q(v_0))$ we have \eqref{odeu0} where $F_0$ is given by
\begin{displaymath}
F_0=\tilde{F}_0-P_0\Big(\dfi{}\Big(v_0(v_1-Q(v_0))\Big)\Big)-3c\e{2}P_0\Big(\dfi{}\Big(v_0(\dfi{}(v_1-Q(v_0)+\dfi{3}(v_1-Q(v_0))))\Big)\Big)
\end{displaymath}
Using \eqref{F0tilde} and the fact that 
\begin{displaymath}
\norm{P_0\Big(\dfi{}\Big(v_0(v_1-Q(v_0))\Big)\Big)+3c\e{2}P_0\Big(\dfi{}\Big(v_0(\dfi{}(v_1-Q(v_0)+\dfi{3}(v_1-Q(v_0))))\Big)\Big)}_{L\e{2}(\T)}\leq C(\norm{v_0}_{\dhil{4}}\norm{v_1-Q(v_0)}_{\dhil{4}})
\end{displaymath}
we have \eqref{F0}.

In the same way, if we apply $P_1$ to \eqref{evoleq} we have,
\begin{equation}\label{phtoeq}
\frac{dv_1}{dt}-L(v_1-Q(v_0))=\tilde{F}_1,
\end{equation}
where we have used \eqref{ldeq} and $\tilde{F}_1$ is given by
\begin{equation}\label{F1tildedef}
\tilde{F}_1=-\dfi{}P_1(v_0v_1)-\dfi{}P_1(\frac{v_1\e{2}}{2})+R_1(v)
\end{equation}
with $$R_1(v)=P_1(\dfi{}((v\e{3}+3c\e{2}v+3cv\e{2})(\dfi{}v_1+\dfi{3}v_1))),$$
since $\dfi{}v_0+\dfi{3}v_0=0$.

It is easy to see that, using \eqref{vhcota}
\begin{align*}
&\norm{\dfi{}P_1(v_0v_1)+\dfi{}P_1(\frac{v_1\e{2}}{2})}_{L\e{2}(\T)}\leq C\norm{v_0}_{\dhil{4}}\norm{v_1}_{\dhil{4}}+C\norm{v_1}\e{2}_{\dhil{4}}\leq C\norm{v_0}_{\dhil{4}}\norm{v_1-Q(v_0)}_{\dhil{4}}\\
&+C\norm{v_0}\norm{Q(v_0)}_{\dhil{4}}+C\norm{v_1}\e{2}_{\dhil{4}}\leq C(\norm{v_0}\e{3}_{\dhil{4}}+\norm{v_1-Q(v_0)}_{\dhil{4}}\e{2}+\norm{v_0}_{\dhil{4}}\norm{v_1-Q(v_0)}_{\dhil{4}})\\
&\norm{R_1(v)}_{L\e{2}(\T)}\leq C\norm{v}\norm{v_1}_{\dhil{4}}\leq C(\norm{v_0}_{\dhil{4}}\e{3}+\norm{v_1-Q(v_0)}_{\dhil{4}}\e{2}+\norm{v_0}_{\dhil{4}}\norm{v_1-Q(v_0)}_{\dhil{4}})
\end{align*}
hence, we have
\begin{equation}\label{F1tilde}
\norm{\tilde{F}_1}_{L\e{2}(\T)}\leq C(\norm{v_0}_{\dhil{4}}\e{3}+\norm{v_1-Q(v_0)}_{\dhil{4}}\e{2}+\norm{v_0}_{\dhil{4}}\norm{v_1-Q(v_0)}_{\dhil{4}})
\end{equation}

We can compute $\frac{d}{dt}(v_1-Q(v_0))$ using the chain rule,
\begin{equation*}
\frac{d}{dt}(v_1-Q(v_0))=\frac{dv_1}{dt}-DQ(v_0)(\frac{dv_0}{dt}),
\end{equation*}
therefore, by \eqref{phtoeq}, we obtain \eqref{uh-q} with $F_1=\tilde{F}_1-DQ(v_0)(\frac{dv_0}{dt})$.

Using $\norm{DQ(v_0)}_{L\e{2}(\T)}\leq C\norm{v_0}_{\dhil{4}}$ as well as \eqref{odeu0} and \eqref{F0} we arrive at
 $$\norm{DQ(v_0)\frac{dv_0}{dt}}_{L\e{2}(\T)}\leq C(\norm{v_0}_{\dhil{4}}\e{4}+\norm{v_0}_{\dhil{4}}\norm{v_1-Q(v_0)}\e{2}_{\dhil{4}}+\norm{v_0}\e{2}_{\dhil{4}}\norm{v_1-Q(v_0)}_{\dhil{4}}).$$ 
 
 Combining this inequality with \eqref{F1tilde}, we obtain \eqref{F1} and the Lemma follows.
\end{proof} 

\begin{lem}\label{regulF1}
Suppose that the assumptions of Lemma \ref{F0F1} are satisfied. There exist $\delta_1\in(0,1)$ depending only on $c$ such that if $\norm{v(\cdot,t)}_{\dhil{4}}\leq\delta_1$ for any $t\in(0,T]$. Then we have that $v\in C\e{\infty}((0,T];\dhil{\ell})$ for all $\ell\geq 1$. Moreover, the following estimate holds:
\begin{equation}\label{derivF1}
\int_{\T}\dfi{4}F_1\dfi{4}(v_1-Q(v_0))d\varphi\leq C(\norm{v}\e{3}_{\dhil{4}}\norm{v_1-Q(v_0)}_{\dhil{6}}+\norm{v}_{\dhil{4}}\norm{v_1-Q(v_0)}\e{2}_{\dhil{6}})\quad\textit{for any}\quad t>0
\end{equation}
with $F_1$ as in \eqref{F1} and $C$ depending only on $c$.
\end{lem}
\begin{proof}
Using that $\norm{v}_{L\e{\infty}(\T)}+\norm{\nabla v}_{L\e{\infty}(\T)}+\norm{\nabla\e{2}v}_{L\e{\infty}(\T)}\leq C\norm{v}_{\dhil{4}}$ it readily follows that \eqref{evoleq} is an uniformly parabolic problem. Hence for the fact that $v\in C\e{\infty}((0,T];\dhil{\ell})$ for all $\ell\geq 1$ just follows by standard regularizing effects for quasilinear parabolic equations, assuming that $\delta_1$ is sufficiently small (cf. \cite{eidelman} in Chapter II, Theorem 4.2.).

As we see in the proof of the Lemma \ref{F0F1}, $F_1=\tilde{F}_1-DQ(v_0)(\frac{dv_0}{dt})$ where $\tilde{F}_1$ is as in \eqref{F1tildedef}. Therefore
\begin{displaymath}
\int_{\T}\dfi{4}F_1\dfi{4}(v_1-Q(v_0))d\varphi=\int_{\T}\dfi{4}\tilde{F}_1\dfi{4}(v_1-Q(v_0))d\varphi-\int_{\T}\dfi{4}(DQ(v_0)(\frac{dv_0}{dt}))\dfi{4}(v_1-Q(v_0))d\varphi\equiv J_1+J_2.
\end{displaymath}

To estimate $J_2$ we remark that $\frac{dv_0}{dt}\in\mathcal{E}_0$ as well as that for any $w\in\mathcal{E}_0$ we have $\norm{w}_{\dhil{\ell}}\leq C_{\ell}\norm{w}_{\dhil{4}}$ for any $\ell\geq 1$. On the other hand, the definition of $Q$ (cf. \eqref{defq}) implies that $\norm{DQ(v_0)}_{\dhil{\ell}}\leq C_{\ell}\norm{v_0}_{\dhil{4}}$ for $\ell\geq 1$. Then, $\norm{\dfi{4}(DQ(v_0)(\frac{dv_0}{dt}))}_{L\e{2}(\T)}\leq C\norm{v_0}_{\dhil{4}}\norm{\frac{dv_0}{dt}}_{\dhil{4}}$. Therefore using \eqref{odeu0} and \eqref{F0}, we obtain 
\begin{align*}
&\norm{\dfi{4}(DQ(v_0)(\frac{dv_0}{dt}))}_{L\e{2}(\T)}\leq C(\norm{v_0}_{\dhil{4}}\e{4}+\norm{v_0}_{\dhil{4}}\norm{v_1-Q(v_0)}_{\dhil{4}}\e{2}+\norm{v_0}\e{2}_{\dhil{4}}\norm{v_1-Q(v_0)}_{\dhil{4}})\\
&\leq C\norm{v}\e{3}_{\dhil{4}}
\end{align*}

Thus, 
\begin{equation}\label{J2}
J_2\leq C(\norm{v}\e{3}_{\dhil{4}}\norm{v_1-Q(v_0)}_{\dhil{4}})
\end{equation}

In order to estimate $J_1$ for $t>0$ we rewrite it, using \eqref{F1tildedef}, as
\begin{align*}
&\int_{\T}\dfi{4}\tilde{F}_1\dfi{4}(v_1-Q(v_0))d\varphi=-\int_{\T}\dfi{5}P_1(v_0v_1)\dfi{4}(v_1-Q(v_0))d\varphi-\int_{\T}\dfi{5}P_1(\frac{v_1\e{2}}{2})\dfi{4}(v_1-Q(v_0))d\varphi\\
&+\int_{\T}\dfi{4}R_1(v)\dfi{4}(v_1-Q(v_0))d\varphi\equiv J_{1,1}+J_{1,2}+J_{1,3}
\end{align*}

We first estimate $J_{1,1}$, using again that for any $w\in\mathcal{E}_0$ we have $\norm{w}_{\dhil{\ell}}\leq C_{\ell}\norm{w}_{\dhil{4}}$ for any $\ell\geq 1$ as well as integration by parts. Then, 

\begin{align*}
&J_{1,1}=\int_{\T}\dfi{5}P_1(v_0Q(v_0))\dfi{4}(v_1-Q(v_0))d\varphi+\int_{\T}\dfi{5}P_1(v_0(v_1-Q(v_0)))\dfi{4}(v_1-Q(v_0))d\varphi\\
&\leq C\norm{v_0}_{\dhil{4}}\e{3}\norm{v_1-Q(v_0)}_{\dhil{4}}+\int_{\T}\dfi{3}P_1(v_0(v_1-Q(v_0))\dfi{6}(v_1-Q(v_0))d\varphi\\
&\leq C(\norm{v_0}\e{3}_{\dhil{4}}\norm{v_1-Q(v_0)}_{\dhil{4}}+\norm{v_0}_{\dhil{4}}\norm{v_1-Q(v_0)}_{\dhil{3}}\norm{v_1-Q(v_0)}_{\dhil{6}})\\
&\leq C(\norm{v}\e{3}_{\dhil{4}}\norm{v_1-Q(v_0)}_{\dhil{6}}+\norm{v}_{\dhil{4}}\norm{v_1-Q(v_0)}_{\dhil{6}}\e{2})
\end{align*}

Similarly,
\begin{align*}
&J_{1,2}=\int_{\T}\dfi{5}(\frac{Q(v_0)\e{2}}{2})\dfi{4}(v_1-Q(v_0))d\varphi+2\int_{\T}\dfi{5}P_1(Q(v_0)(v_1-Q(v_0)))\dfi{4}(v_1-Q(v_0))d\varphi\\
&+\int_{\T}\dfi{5}P_1((v_1-Q(v_0))\e{2})\dfi{4}(v_1-Q(v_0))d\varphi\leq C\norm{v_0}_{\dhil{4}}\e{4}\norm{v_1-Q(v_0)}_{\dhil{4}}\\
&+2\int_{\T}\dfi{3}P_1(Q(v_0)(v_1-Q(v_0)))\dfi{6}(v_1-Q(v_0))d\varphi+\int_{\T}\dfi{3}P_1((v_1-Q(v_0))\e{2})\dfi{6}(v_1-Q(v_0))d\varphi\\
&\leq C(\norm{v_0}_{\dhil{4}}\e{4}\norm{v_1-Q(v_0)}_{\dhil{4}}+\norm{v_0}_{\dhil{4}}\e{2}\norm{v_1-Q(v_0)}_{\dhil{3}}\norm{v_1-Q(v_0)}_{\dhil{6}}\\
&+\norm{v_1-Q(v_0)}_{\dhil{4}}\e{2}\norm{v_1-Q(v_0)}_{\dhil{6}})\leq C(\norm{v}\e{3}_{\dhil{4}}\norm{v_1-Q(v_0)}_{\dhil{6}}+\norm{v}_{\dhil{4}}\norm{v_1-Q(v_0)}_{\dhil{6}}\e{2})
\end{align*}

Integrating by parts in $J_{1,3}$, we arrive at
\begin{align*}
&J_{1,3}=\int_{\T}\dfi{3}P_1((v\e{3}+3c\e{2}v+3cv\e{2})(\dfi{}Q(v_0)+\dfi{3}Q(v_0)))\dfi{6}(v_1-Q(v_0))d\varphi\\
&+\int_{\T}\dfi{3}P_1((v\e{3}+3c\e{2}v+3cv\e{2})(\dfi{}(v_1-Q(v_0))+\dfi{3}(v_1-Q(v_0))))\dfi{6}(v_1-Q(v_0))d\varphi\equiv J_{1,3,1}+J_{1,3,2}
\end{align*}

Then
\begin{displaymath}
J_{1,3,1}\leq C\norm{v}_{\dhil{4}}\norm{v_0}\e{2}_{\dhil{4}}\norm{v_1-Q(v_0)}_{\dhil{6}}
\end{displaymath}
and 
\begin{displaymath}
J_{1,3,2}\leq C\norm{v}_{\dhil{4}}\norm{v_1-Q(v_0)}\e{2}_{\dhil{6}}
\end{displaymath}
whence 
\begin{displaymath}
J_{1,3}\leq C(\norm{v}\e{3}_{\dhil{4}}\norm{v_1-Q(v_0)}_{\dhil{6}}+\norm{v}_{\dhil{4}}\norm{v_1-Q(v_0)}\e{2}_{\dhil{6}})
\end{displaymath}

Therefore 
\begin{displaymath}
J_{1}=J_{1,1}+J_{1,2}+J_{1,3}\leq C(\norm{v}\e{3}_{\dhil{4}}\norm{v_1-Q(v_0)}_{\dhil{6}}+\norm{v}_{\dhil{4}}\norm{v_1-Q(v_0)}\e{2}_{\dhil{6}})
\end{displaymath}

Combining this estimate with \eqref{J2} and the result follows.
\end{proof}

We can now conclude the proof of Theorem \ref{thglobal}.

\begin{proof}[Proof of Theorem \ref{thglobal}]
Let $\norm{v_{in}}_{\dhil{4}}\leq \varepsilon$. If $\varepsilon>0$ is sufficiently small it follows from Proposition \ref{wellposedness} that $\norm{v(\cdot,t)}_{\dhil{4}}\leq \delta_1$ for $t\in(0,T]$. Then Lemma \ref{regulF1} implies that $v\in C\e{\infty}((0,T];\dhil{8})$. Moreover, using \eqref{cotadatoinicial} we can assert that 
\begin{equation}\label{cotadeepsilon}
\norm{v(\cdot,t)}_{\dhil{4}}\leq C_*\varepsilon\quad\textit{for}\quad t\in[0,T]
\end{equation}
where we can assume that $C_*\geq 4K$ for $K=\sqrt{12c\e{3}}$.

Differentiating \eqref{uh-q} four times with respect to $\varphi$ for $t>0$ and multiplying by $\dfi{4}(v_1-Q(v_0))$, we obtain using integration by parts
\begin{align*}
&\frac{d}{dt}\norm{v_1-Q(v_0)}_{\dhil{4}}\e{2}+c\e{3}\int_{\T}(\dfi{6}(v_1-Q(v_0)))\e{2}-(\dfi{5}(v_1-Q(v_0)))\e{2}d\varphi=\int_{\T}\dfi{4}F_1\dfi{4}(v_1-Q(v_0))d\varphi
\end{align*}

Using Fourier as well as the fact that $v_1-Q(v_0)\in\mathcal{E}_1$ we obtain 
\begin{displaymath}
c\e{3}\int_{\T}(\dfi{6}(v_1-Q(v_0)))\e{2}-(\dfi{5}(v_1-Q(v_0)))\e{2}d\varphi\geq C_0\norm{v_1-\psi(v_0)}_{\dhil{6}}\e{2},
\end{displaymath}
for some $C_0>0$ (depending on $c$).

Using Lemma \ref{regulF1} we then arrive at
\begin{align*}
&\frac{d}{dt}\norm{v_1-Q(v_0)}_{\dhil{4}}\e{2}\leq -C_0\norm{v_1-Q(v_0)}_{\dhil{6}}\e{2}+C(\norm{v}\e{3}_{\dhil{4}}\norm{v_1-Q(v_0)}_{\dhil{6}}+\norm{v}_{\dhil{4}}\norm{v_1-Q(v_0)}\e{2}_{\dhil{6}})
\end{align*}
and using Young and \eqref{vhcota} we deduce that
\begin{align*}
&\frac{d}{dt}\norm{v_1-Q(v_0)}_{\dhil{4}}\e{2}\leq -\frac{C_0}{2}\norm{v_1-Q(v_0)}_{\dhil{6}}\e{2}+C_1\norm{v_0}\e{6}_{\dhil{4}}+C_2\norm{v}_{\dhil{4}}\norm{v_1-Q(v_0)}\e{2}_{\dhil{6}}.
\end{align*}

We define $\delta=\min\{\delta_0,\delta_1,\frac{C_0}{4C_2}\}$ with $\delta_0$, $\delta_1$ as in Proposition \ref{wellposedness} and Lemma \ref{regulF1}, respectively. Choosing $\varepsilon$ sufficiently small we obtain that $\norm{v(\cdot,t)}_{\dhil{4}}\leq\delta$ for $0\leq t\leq T$. Then, 
\begin{equation}\label{cotaevolv1-q}
\frac{d}{dt}\norm{v_1-Q(v_0)}_{\dhil{4}}\e{2}\leq -\frac{C_0}{4}\norm{v_1-Q(v_0)}_{\dhil{4}}\e{2}+C_1\norm{v_0}\e{6}_{\dhil{4}}.
\end{equation}

Notice that if the solution $v$ can be extended to a larger time interval, the proof of \eqref{cotaevolv1-q} implies that this inequality is valid as long as $\norm{v}_{\dhil{4}}\leq \delta$. Thus, using that $\norm{v_{in}}_{\dhil{4}}\leq \varepsilon$, we have
\begin{equation}\label{uh-psiest}
\norm{v_1-Q(v_0)}_{\dhil{4}}\e{2}\leq C_3\varepsilon\e{2} e\e{-\frac{C_0}{2}t}+C_1\int_{0}\e{t}\norm{v_0}_{\dhil{4}}\e{6}(s)e\e{-\frac{C_0}{4}(t-s)}ds
\end{equation}

Now using \eqref{defq} we can readily see that
\begin{align*}
&P_0\Big(\dfi{}\Big(v_0Q(v_0)\Big)\Big)=\frac{1}{24c^3}\Big(a_1 a_{-1}^2e^{-i\varphi}+a_{-1}a_{1}^2e^{i\varphi}\Big),\\
&P_0\Big(3c^2\dfi{}\Big(v_0(\dfi{}Q(v_0)+\dfi{3}Q(v_0))\Big)\Big)=\frac{3i}{4c}\Big(a_{1}a_{-1}^2e^{-i\varphi}-a_{-1}a_{1}^2e^{i\varphi}\Big).
\end{align*}
therefore we can write \eqref{odeu0} as
\begin{align*}
&a'_1(t)e^{i\varphi}+a'_{-1}(t)e^{-i\varphi}=\Big(-\frac{1}{24c^3}-\frac{3i}{4c}\Big)a_1 a_{-1}^2 e^{-i\varphi}+\Big(-\frac{1}{24c^3}+\frac{3i}{4c}\Big)a_{-1}a_1^2 e^{i\varphi}\\
&+\mathcal{O}(\norm{F_0}_{L\e{2}(\T)}).
\end{align*}

Since $v_0\in\R$ it is clear that $a_{-1}=\overline{a_1}$ and using Lemma \ref{F0F1}, we obtain
\begin{equation}\label{ecuaciona1}
a'_1(t)+\alpha\abs{a_1}^2 a_{1}\leq C_4(\norm{v_0}^4_{\dhil{4}}+\norm{v_1-Q(v_0)}\e{2}_{\dhil{4}}+\norm{v_0}_{\dhil{4}}\norm{v_1-Q(v_0)}_{\dhil{4}})
\end{equation}
where $\alpha=\Big(\frac{1}{24c^3}-\frac{3i}{4c}\Big)$. 
%

We now define 
\begin{equation}\label{testrella}
t_*=\sup\{t\geq 0: \norm{v_0(\cdot,s)}_{\dhil{4}}\leq M\varepsilon\},\quad\textit{for all}\quad s\in[0,t]
\end{equation}
where $M> 2$ (we could assume for instance $M=3$ in all the following). Using \eqref{cotadeepsilon} it follows that choosing $\varepsilon$ sufficiently small, we have $t_*\geq T$.

Using the polar representation $a_1(t)=\rho(t)e^{i\gamma(t)}$  and using \eqref{uh-psiest}, \eqref{ecuaciona1} and \eqref{testrella}, we arrive at
\begin{align*}
&\abs{\rho'(t)+Re(\alpha)\rho^3(t)}\leq C_4(M\e{4}\varepsilon\e{4}+C_3\varepsilon\e{2}e\e{-\frac{C_0}{2}t}+\frac{4C_1}{C_0}M\e{6}\varepsilon\e{6}+\sqrt{C_3}M\varepsilon\e{2}e\e{-\frac{C_0}{4}t}+\Big(\frac{4C_1}{C_0}\Big)\e{\frac{1}{2}}M\e{4}\varepsilon\e{4}\Big)
\end{align*}
whence
\begin{equation}\label{ecuacionro}
\rho'(t)\leq -Re(\alpha)\rho\e{3}(t)+C_5 M\e{4}\varepsilon\e{4}+C_6 M\varepsilon\e{2}e\e{-\frac{C_0}{4}t}\quad\textit{for}\quad 0\leq t\leq t_*,
\end{equation}
where $C_5$ and $C_6$ depend on $C_1, C_3$ and $C_4$.

Suppose that $t_*<\infty$. Note that $\abs{a_1}=\rho(t)=\frac{\norm{v_0}_{L\e{2}(\T)}}{\sqrt{4\pi}}\leq\norm{v_0}_{\dhil{4}}$. Therefore since $\rho(t)$ is a continuous function, there exists $\bar{t}\in(0,t_*)$ such that $\rho(\bar{t})=\frac{M}{2}\varepsilon$ and $\frac{M}{2}\varepsilon\leq\rho(t)\leq M\varepsilon$ for $t\in[\bar{t},t_*]$. Thus \eqref{ecuacionro} implies that, for $\varepsilon\leq\frac{Re(\alpha)}{8MC_5}$,
\begin{displaymath}
\rho'(t)\leq -Re(\alpha)\frac{M\e{3}}{8}\varepsilon\e{3}+C_5 M\e{4}\varepsilon\e{4}+C_6 M\varepsilon\e{2}e\e{-\frac{C_0}{4}t}\leq C_6 M\varepsilon\e{2}e\e{-\frac{C_0}{4}t} 
\end{displaymath}
for $t\in[\bar{t},t_*]$.

Hence, 
\begin{displaymath}
\rho(t)\leq \frac{M}{2}\varepsilon+\frac{4C_6}{C_0}M\varepsilon\e{2}\leq \frac{2M\varepsilon}{3}\quad\textit{for}\quad t\in[\bar{t},t_*],
\end{displaymath}
if $\varepsilon$ is sufficiently small depending only on $C_0$ and $C_6$. However, this contradicts the definition of $t_*$ in \eqref{testrella}.

Then, $\norm{v_0}_{\dhil{4}}\leq C\varepsilon$ for   $0\leq t<\infty$ and \eqref{uh-psiest} implies $\norm{v_1-Q(v_0)}_{\dhil{4}}\leq C\varepsilon$. Thus, the Theorem \ref{thglobal} follows.

\end{proof}

A more detailed characterization of the asymptotic behaviour of $h$ can be obtain using Center Manifold Theory. A version of this theory that can be applied to quasilinear systems, has been developed in \cite{Mielke}. We will use the version of this theory that can be found in \cite{centermanif}.

We recall the assumptions required to apply the Center Manifold Theory in \cite{centermanif}.  Actually, we adapt them to the particular functional setting that we use in this paper:
\begin{hip}\label{hip1}
$L$ and $R$ defined above have the following properties:
\begin{itemize}
\item[(i)]$L\in\mathcal{L}(\dot{H}^4(\T),L\e{2}(\T))$.
\item[(ii)]For some $k\geq 2$, there exists a neighbourhood $\mathcal{V}\subset\dot{H}^4(\T)$ of $0$ such that $R\in\mathcal{C}^k(\mathcal{V};L\e{2}(\T))$, $R(0)=0$ and $DR(0)=0$.
\end{itemize}
\end{hip}
\begin{hip}\label{hip2}
The spectrum $\sigma$ of the linear operator $L$ can be written as $\sigma=\sigma_-\cup\sigma_0$ where $\sigma_-=\{\lambda\in\sigma; Re\lambda<0\}$ and $\sigma_0=\{\lambda\in\sigma; Re\lambda=0\}$. We assume that
\begin{itemize}
\item[(i)]there exist a positive constant $\gamma>0$ such that $$\sup_{\lambda\in\sigma_-}(Re\lambda)<-\gamma,$$ 
\item[(ii)]$\sigma_0$ consists of a finite number of eigenvalues with finite multiplicities. 
\end{itemize}
\end{hip}

\begin{hip}\label{hip3}
Assume that there exist positive constants $ s_0>0$ and $C>0$ such that for all $ s\in\R$, with $\abs{s}\geq s_0$, we have that
\begin{equation}\label{esthip3}
\norm{(i s I-L_1)^{-1}}_{\mathcal{L}(L\e{2}(\T))}\leq\frac{C}{\abs{ s}}.
\end{equation}

Here, $L_1$ is the restriction of $L$ to $P_1\dot{H}^4(\T)$ where $P_1$ is the projection $P_1:L\e{2}(\T)\to L\e{2}(\T)$ defined by $P_1=\mathbb{I}-P_0$ where $P_0$ is the spectral projection corresponding to $\sigma_0$ that is given by:
\begin{equation}\label{defp0}
P_0=\frac{1}{2\pi i}\int_{\Gamma}(\lambda\mathbb{I}-L)\e{-1}d\lambda
\end{equation}
where $\Gamma$ is a simple, counterclockwise oriented, Jordan curve surrounding $\sigma_0$ and lying entirely in $\{\lambda\in\C:Re\lambda>-\gamma\}$.
\end{hip}
The following result is a minor adaptation of Theorems 2.9 and 3.22 in \cite{centermanif}. It is important to note that since we are working with the Hilbert spaces $L\e{2}(\T)$ and $\dot{H}\e{4}(\T)$ the only property that we need to check for the operator $L_1$ is \eqref{esthip3}, due to Remark 2.16 in \cite{centermanif}.
\begin{thm}\label{thmmani}
Assume that Hypotheses \ref{hip1}, \ref{hip2} and \ref{hip3} hold. Then there exists a map $\psi\in\mathcal{C}\e{k}(\mathcal{E}_0,\mathcal{E}_1)$ where $\mathcal{E}_0=ImP_0=ReP_1\subset\dot{H}^4(\T)$ and $\mathcal{E}_1=P_1\dot{H}^4\subset\dot{H}^4(\T)$, with $\psi(0)=0$ and $D\psi(0)=0$.
Moreover there exists a neighbourhood $\mathcal{O}$ of $0$ in $\dot{H}^4(\T)$ such that the manifold 
\begin{equation}\label{manifold}
\mathcal{M}_0=\{v_0+\psi(v_0);v_0\in\mathcal{E}_0\}\subset\dot{H}^4(\T)
\end{equation} has the following properties:
\begin{itemize}
\item[(i)]$\mathcal{M}_0$ is locally invariant, i.e., if $v$ is a solution of \eqref{evoleq} satisfying $v(0)\in\mathcal{M}_0\cap\mathcal{O}$
and $v(t)\in\mathcal{O}$ for all $t\in[0,T]$, then $v(t)\in\mathcal{M}_0$ for all $t\in[0,T]$.
\item[(ii)] $\mathcal{M}_0$ contains the set of bounded solutions of  \eqref{evoleq} staying in $\mathcal{O}$ for all $t\in\R$. If we have a solution $v$ of $\frac{dv}{dt}=L(v)+R(v)$ that belongs in $\mathcal{M}_0$ for $t\in I$, being $I\subset\R$ an open interval. Then, $v=v_0+\psi(v_0)$ and $v_0$ satisfies 
\begin{equation}\label{v0}
\frac{dv_0}{dt}=L_0(v_0)+P_0R(v_0+\psi(v_0)),
\end{equation}
where $L_0$ is the restriction of $L$ to $\mathcal{E}_0$.
Moreover, $\psi$ satisfies 
\begin{equation}\label{psi}
D\psi(v_0)(L_0(v_0)+P_0R(v_0+\psi(v_0)))=L_1\psi(v_0)+P_1R(v_0+\psi(v_0))\quad\forall v_0\in\mathcal{E}_0
\end{equation}
\item[(iii)] $\mathcal{M}_0$ is locally attracting, i.e., there exists $a>0$ such that if $v(0)\in\mathcal{O}$ and the solution for this initial data of \eqref{evoleq} satisfies that $v(t)\in\mathcal{O}$ for all $t>0$, then exist a initial data $\tilde{v}(0)\in\mathcal{M}_0\cap\mathcal{O}$ such that, $\norm{v-\tilde{v}}_{\dot{H}\e{4}(\T)}\leq Ce\e{-at}$ as $t\to\infty$.
\end{itemize} 

\end{thm}
\begin{nota}
Notice that the subspaces $\mathcal{E}_0$ and $\mathcal{E}_1$ have been defined in \eqref{e0e1} in a different way as in the statement of Theorem \ref{thmmani}. However, using Fourier analysis and  \eqref{defp0} it might be readily seen that both definitions are equivalent.
\end{nota}

In the following Lemma we collect several properties of the operators
 $L$ and $R$ in \eqref{defl} and \eqref{defr}. In particular they satisfy the Hypotheses \ref{hip1}, \ref{hip2} and \ref{hip3}.
\begin{lem}\label{lema1}
The operators $L$ and $R$ given by \eqref{defl} and \eqref{defr} are well defined from $\dot{H}\e{4}(\T)\to L\e{2}(\T)$ and $\dot{H}\e{4}(\T)\times\dot{H}\e{4}(\T)\to L\e{2}(\T)$ respectively. 

Moreover, the operator $L$ satisfies Hypotheses \ref{hip1}(i), \ref{hip2} and \ref{hip3} and the operator $R$ satisfies Hypothesis \ref{hip1}(ii) for any $k\geq 2$. 
\end{lem}

\begin{proof}

It is readily seen that \eqref{defl} defines an operator $L$ from $\dot{H}\e{4}(\T)\to L\e{2}(\T)$ for any $c>0$ and $L\in\mathcal{L}(\dot{H}\e{4}(\T),L\e{2}(\T))$ whence Hypothesis \ref{hip1}(i) follows. 

In order to check Hypothesis \ref{hip2}, we use the Fourier representation of $L$. Given $v=\sum_{n=-\infty, n\neq 0}^{n=\infty}a_n e^{in\varphi}\in\dot{H}\e{4}(\T)$ and $Lv=\sum_{n=-\infty, n\neq 0}^{n=\infty}(Lv)_n e^{in\varphi}\in L\e{2}(\T)$ we have
\begin{equation}\label{fourierl}
(Lv)_n=-c^3(n^4-n^2)a_n\quad\textit{for}\quad n\in\Z\setminus\{0\}.
\end{equation}

It follows that $\sigma=\sigma(L)=\sigma_p(L)=\{\lambda=-c^3(n^4-n^2): n\in\N\}$. 
Therefore, $\sigma=\sigma_0\cup\sigma_+$ with 
\begin{displaymath}
\sigma_0=\{0\},\quad\sigma_-=\{\lambda=-c^3(n^4-n^2):n=2,3,...\} 
\end{displaymath}
whence Hypothesis \ref{hip2} follows.

In order to check Hypothesis \ref{hip3} we first remark that $P_0$ is the orthogonal projection in $L\e{2}(\T)$ in the subspace of eigenvectors associated to $\sigma_0$ (cf. \cite{kato} and \cite{reed_simon}).
We now remark that \eqref{esthip3} would follow from the estimate 
\begin{displaymath}
\norm{i s v-L_1(v)}_{L\e{2}(\T)}\geq\abs{ s}\norm{v}_{L\e{2}(\T)}\quad\textit{for any}\quad s\in\R, \abs{s}\geq s_0\quad\textit{and}\quad v\in\mathcal{E}_1.
\end{displaymath}
This estimate is a consequence of the following computation,
\begin{align*}
&\norm{i sv+c^3(\dfi{2}v+\dfi{4}v)}_{L^2}^2=\int_{\T}\abs{i sv+c^3(\dfi{2}v+\dfi{4}v)}^2d\varphi=\int_{\T} s^2v^2d\varphi+c^6\int_\T(\dfi{2}v+\dfi{4}v)^2d\varphi\geq s^2\norm{v}_{L\e{2}(\T)}^2,
\end{align*}
which holds for any $v\in\mathcal{E}_1$.

Concerning the operator defined in \eqref{defr}, we first notice that it is a well defined operator from $\dot{H}\e{4}(\T)\times\dot{H}\e{4}(\T)\to L\e{2}(\T)$, due to the embedding $\dot{H}\e{1}(\T)\subset L\e{\infty}(\T)$. It only remains to check Hypothesis \ref{hip1}(ii).

It is enough to see that 
\begin{equation}\label{rck}
\norm{R(v+\varepsilon)-\sum_{\ell=0}^{k}\frac{1}{\ell !}D^\ell R(v)(\varepsilon)}_{L\e{2}(\T)}\leq C\norm{\varepsilon}_{\dot{H}^4}\e{k+1}\quad\textit{for}\quad k\geq 2
\end{equation}
where for any $v\in\dot{H}\e{4}(\T)$, $D\e{\ell}R(v)\in\mathcal{L}((\dot{H}\e{4}(\T))\e{\ell},L\e{2}(\T))$.

Inequality \eqref{rck} is a consequence of the fact that 
\begin{displaymath}
R(v+\varepsilon)=\sum_{\ell=0}^{4}\frac{1}{\ell !}D^\ell R(v)(\varepsilon)
\end{displaymath}
where 
\begin{align}\label{derivadasder}
&DR(v)(\varepsilon)=-\dfi{}(\varepsilon v)-\dfi{}((3c\e{2}\varepsilon+6cv\varepsilon+3v\e{2}\varepsilon)(\dfi{}v+\dfi{3}v))-\dfi{}((v\e{3}+3c\e{2}v+3cv\e{2})(\dfi{}\varepsilon+\dfi{3}\varepsilon)),\\
&\frac{1}{2}D\e{2}R(v)(\varepsilon)=-\dfi{}(\frac{\varepsilon\e{2}}{2})-\dfi{}((3c\e{2}\varepsilon+6cv\varepsilon+3v\e{2}\varepsilon)(\dfi{}\varepsilon+\dfi{3}\varepsilon))-\dfi{}((3c\varepsilon\e{2}+3v\varepsilon\e{2})(\dfi{}v+\dfi{3}v)),\\
&\frac{1}{3!}D\e{3}R(v)(\varepsilon)=-\dfi{}((3c\varepsilon\e{2}+3v\varepsilon\e{2})(\dfi{}\varepsilon+\dfi{3}\varepsilon))-\dfi{}(\varepsilon\e{3}(\dfi{}v+\dfi{3}v)),\\
&\frac{1}{4!}D\e{4}R(v)(\varepsilon)=-\dfi{}(\varepsilon\e{3}(\dfi{}\varepsilon+\dfi{3}\varepsilon)).
\end{align}

Using classical Sobolev embeddings, in particular the fact that the Sobolev spaces $\dhil{4}$ are Banach algebras (cf. \cite{adams_fournier}) we obtain that $D\e{\ell}R(v)(\cdot)\in\mathcal{L}((\dot{H}\e{4}(\T))\e{\ell},L\e{2}(\T))$, $\ell=0,1,2,...$ whence \eqref{rck} holds for all $k\geq 2$.

We suppose that $v\in\dot{H}^4(\T)$ then we have the following estimates:
\begin{align*}
&\norm{DR(v)(\varepsilon)}_{L\e{2}(\T)}\leq\norm{\dfi{}(\varepsilon v)}_{L^2}+c^2\norm{\dfi{}(\varepsilon(\dfi{}v+\dfi{3}v))}_{L^2}\\
&+2c\norm{\dfi{}(v\varepsilon(\dfi{}v+\dfi{3}v))}_{L^2}+3\norm{\dfi{}(v^2\varepsilon(\dfi{}v+\dfi{3}v))}_{L^2}\\
&+\norm{\dfi{}((v^3+c^2v+cv^2)(\dfi{}\varepsilon+\dfi{3}\varepsilon))}_{L^2}\equiv\sum_{i=1}^5 I_i^1
\end{align*}
Using Holder inequality we can obtain,
\begin{align*}
&I^1_1\leq\norm{\dfi{}\varepsilon v}_{L^2}+\norm{\varepsilon\dfi{}v}_{L^2}\leq\norm{v}_{L^{\infty}}\norm{\varepsilon}_{\dot{H}^1}+\norm{\dfi{}v}_{L^{\infty}}\norm{\varepsilon}_{L\e{2}}\leq C\norm{\varepsilon}_{\dot{H}^1};
\end{align*}
\begin{align*}
&I^1_2\leq c^2(\norm{\dfi{}\varepsilon(\dfi{}v+\dfi{3}v)}_{L^2}+\norm{\varepsilon(\dfi{2}v+\dfi{4}v)}_{L^2}\leq c^2((\norm{\dfi{}v}_{L^{\infty}}+\norm{\dfi{3}v}_{L^{\infty}})\norm{\varepsilon}_{\dot{H}^1}\\
&+\norm{\dfi{2}v}_{L^{\infty}}\norm{\varepsilon}_{L\e{2}}+\norm{v}_{\dot{H}^4}\norm{\varepsilon}_{L^{\infty}})\leq C\norm{\varepsilon}_{\dot{H}^1};
\end{align*}
\begin{displaymath}
I^1_3\leq 2c(\norm{\dfi{}(\varepsilon v\dfi{}v)}_{L^2}+\norm{\dfi{}(\varepsilon v\dfi{3}v)}_{L^2})
\end{displaymath}
where
\begin{align*}
&\norm{\dfi{}(\varepsilon v\dfi{}v)}_{L^2}\leq \norm{\dfi{}\varepsilon v\dfi{}v}_{L^2}+\norm{\varepsilon\dfi{}v\dfi{}v}_{L^2}+\norm{\varepsilon v\dfi{2}v}_{L^2}\\
&\leq C(\norm{v}_{L\e{\infty}}\norm{\dfi{}v}_{L^{\infty}}\norm{\varepsilon}_{\dot{H}^1}+\norm{\dfi{}v}_{L^{\infty}}^2\norm{\varepsilon}_{L\e{2}}+\norm{v}_{L\e{\infty}}\norm{\dfi{2}v}^2_{L^{\infty}}\norm{\varepsilon}_{L^2})\leq C\norm{\varepsilon}_{\dot{H}^1}
\end{align*}
and
\begin{align*}
&\norm{\dfi{}(\varepsilon v\dfi{3}v)}_{L^2}=\norm{(\dfi{}\varepsilon v+\varepsilon\dfi{}v)\dfi{3}v+\varepsilon v\dfi{4}v}_{L^2}\\
&\leq\norm{v}_{L^{\infty}}\norm{\dfi{3}v}_{L^{\infty}}\norm{\varepsilon}_{\dot{H}^1}+\norm{\dfi{}v}_{L^{\infty}}\norm{\dfi{3}v}_{L^{\infty}}\norm{\varepsilon}_{L\e{2}}+\norm{v}_{L\e{\infty}}\norm{v}_{\dot{H}\e{4}}\norm{\varepsilon}_{L^\infty}\leq C\norm{\varepsilon}_{\dot{H}^1}
\end{align*}
Following the same procedure with the terms $I^1_4$ and $I^1_5$ we get $\norm{DR(v)(\varepsilon)}_{L\e{2}}\leq C\norm{\varepsilon}_{\dot{H}^4}$.

For $D\e{2}R(v)(\varepsilon)$ we proceed in a similar way:
\begin{align*}
&\norm{D\e{2}R(v)(\varepsilon)}_{L\e{2}}\leq \norm{\dfi{}(\frac{\varepsilon\e{2}}{2})}_{L\e{2}}+\norm{\dfi{}((3c\e{2}\varepsilon+6cv\varepsilon+3v\e{2}\varepsilon)(\dfi{}\varepsilon+\dfi{3}\varepsilon))}_{L\e{2}}+\norm{\dfi{}((3c\varepsilon\e{2}+3v\varepsilon\e{2})(\dfi{}v+\dfi{3}v))}_{L\e{2}}\\
&=\norm{\varepsilon\dfi{}\varepsilon}_{L\e{2}}+\norm{\dfi{}(3c\e{2}\varepsilon+6cv\varepsilon+3v\e{2}\varepsilon)}_{L\e{2}}\norm{(\dfi{}\varepsilon+\dfi{3}\varepsilon)}_{L\e{\infty}}+\norm{(3c\e{2}\varepsilon+6cv\varepsilon+3v\e{2}\varepsilon)}_{L\e{\infty}}\norm{(\dfi{2}\varepsilon+\dfi{4}\varepsilon)}_{L\e{2}}\\
&+\norm{\dfi{}(3c\varepsilon\e{2}+3v\varepsilon\e{2})}_{L\e{2}}\norm{(\dfi{}v+\dfi{3}v)}_{L\e{\infty}}+\norm{(3c\varepsilon\e{2}+3v\varepsilon\e{2})}_{L\e{\infty}}\norm{(\dfi{2}v+\dfi{4}v)}_{L\e{2}}\leq C\norm{\varepsilon}_{\dot{H}\e{4}}\e{2}.
\end{align*}
In the same way we can compute that 
\begin{align*}
&\norm{D\e{3}R(v)(\varepsilon)}_{L\e{2}}\leq C\norm{\varepsilon}_{\dot{H}\e{4}}\e{3},\\
&\norm{D\e{4}R(v)(\varepsilon)}_{L\e{2}}\leq C\norm{\varepsilon}_{\dot{H}\e{4}}\e{4}
\end{align*}
thus we conclude the proof.
\end{proof}
Now, Lemma \ref{lema1} shows that the assumptions in Theorem \ref{thmmani} hold. In particular this implies the existence of a Center Manifold whose properties are summarised in the following result. 
\begin{lem}\label{lema2}
Suppose that $L$ and $R$ are as in Lemma \ref{lema1}. Then the operator $P_0$ defined in \eqref{defp0} is the orthogonal projection of $L\e{2}(\T)$ in $\mathcal{E}_0=Span\{cos(\varphi),\sin(\varphi)\}$. There exist a manifold $\mathcal{M}_0$ with the properties stated in Theorem \ref{thmmani} that can be parametrized as in \eqref{manifold} with $\psi\in\mathcal{C}\e{k}(\mathcal{E}_0,\mathcal{E}_1)$ for any $k\geq 2$. Moreover, we have $\psi(0)=0$, $D\psi(0)=0$ and if $v_0=a_{-1}e\e{-i\varphi}+a_1e\e{i\varphi}\in\mathcal{E}_0$ then
\begin{equation}
D^2\psi(0)(v_0,v_0)=\frac{i}{12c^3}a_{-1}^2 e^{-2i\varphi}-\frac{i}{12c^3}a_{1}^2 e^{2i\varphi}.
\end{equation}
\end{lem}
\begin{proof}

We have seen in the proof of Lemma \ref{lema1} that $P_0$ is the orthogonal projection of $L\e{2}(\T)$ into the kernel of $L$ which due to \eqref{fourierl} is given by $Span\{cos(\varphi),\sin(\varphi)\}=\mathcal{E}_0$ (cf. Theorem \ref{thmmani}).

The existence of the Center Manifold $\mathcal{M}_0$, is then a consequence of Theorem \ref{thmmani} with the form \eqref{manifold}, where $\psi\in\mathcal{C}\e{k}(\mathcal{E}_0,\mathcal{E}_1)$ satisfies \eqref{psi}.

In the rest of the proof of this Lemma, we will write $\norm{\cdot}=\norm{\cdot}_{\dot{H}\e{4}(\T)}$ in order to simplify the notation. We now use \eqref{psi} to compute $D^2\psi(0)$.
Taking into account that $\psi(0)=D\psi(0)=0$, we obtain using Taylor series at $0$ for $\psi$ and $D\psi$,
\begin{equation}\label{psiv0}
\psi(v_0)=\frac{1}{2}D^2\psi(0)(v_0,v_0)+\mathcal{O}(\norm{v_0}^3)\quad\textit{for any}\quad v_0\in\mathcal{E}_0
\end{equation}
and
\begin{displaymath}
D\psi(v_0)(h)=D^2\psi(0)(v_0,h)+\mathcal{O}(\norm{v_0}^2\norm{h})\quad\textit{for any}\quad h\in\mathcal{E}_0.
\end{displaymath}
Using Taylor series for $R$, as well as $R(0)=DR(0)=0$ we have $R(v_0)=\frac{1}{2}D^2R(0)(v_0,v_0)+\mathcal{O}(\norm{v_0}^3)$. Combining this with \eqref{psiv0}, we obtain 

\begin{displaymath}
R(v_0+\psi(v_0))=\frac{1}{2}D^2R(0)(v_0,v_0)+\mathcal{O}(\norm{v_0}^3),
\end{displaymath} 

Since $L_0(v_0)=0$, equation (\ref{psi}) can be written as:
\begin{displaymath}
L_1(D^2\psi(0)(v_0,v_0))+ P_1(D^2R(0)(v_0,v_0))+\mathcal{O}(\norm{v_0}^3)=0.
\end{displaymath}

Taking $v_0=\varepsilon\xi_0$ with $\xi_0\in\mathcal{E}_0$, dividing by $\varepsilon\e{2}$ and taking the limit $\varepsilon\to 0$ we arrive at
\begin{equation}\label{77taylor}
L_1(D^2\psi(0)(\xi_0,\xi_0))+ P_1(D^2R(0)(\xi_0,\xi_0)))=0.
\end{equation}
 We recall that $D^2R(0)(v,v)=-\dfi{}(\frac{v^2}{2})-3c^2\dfi{}(v(\dfi{}v+\dfi{3}v))$ (cf. \eqref{derivadasder}), that for $v=\xi_0\in\mathcal{E}_0$ reduces to $D^2R(0)(\xi_0,\xi_0)=-\dfi{}(\frac{\xi_0^2}{2})$. Therefore \eqref{77taylor} becomes
\begin{displaymath}
L_1(D^2\psi(0)(\xi_0,\xi_0))=P_1(\dfi{}(\frac{\xi_0^2}{2}))
\end{displaymath}
whence,
\begin{displaymath}
D^2\psi(0)(\xi_0,\xi_0)=L_1\e{-1}P_1(\dfi{}(\frac{\xi_0^2}{2}))
\end{displaymath}
using the fact that $L_1$ is invertible in $P_1L\e{2}(\T)$. Taking $\xi_0=a_{-1}e^{-i\varphi}+a_{1}e^{i\varphi}$ with $a_1=\overline{a_{-1}}$, then
\begin{displaymath}
P_1(\dfi{}(\frac{\xi_0^2}{2}))=i(a_1^2 e^{2i\varphi}-a_{-1}^2 e^{-2i\varphi}).
\end{displaymath}
Finally, we invert $L_1$ using \eqref{fourierl} and we obtain,

\begin{equation}\label{f}
D^2\psi(0)(v_0,v_0)=\frac{i}{12c^3}a_{-1}^2 e^{-2i\varphi}-\frac{i}{12c^3}a_{1}^2 e^{2i\varphi}.
\end{equation}
\end{proof}

\begin{thm}\label{thestab}
Let $c>0$. There exists  $\varepsilon>0$ and a manifold $\mathcal{M}_0$ as in \eqref{manifold} (both of them depending on $c$) such that all the properties stated in Theorem \ref{thmmani} hold with $\mathcal{O}=B_{\varepsilon}(0)$. In particular, if $v(\cdot,0)\in\mathcal{M}_0\cap\mathcal{O}$ the corresponding function $h$ which solves \eqref{equevol} (cf. \eqref{defvfuncionh}) satisfies:
\begin{equation}\label{cotasolucmanifold}
\norm{h(\cdot+ct,t)-c-\frac{2K}{\sqrt{t}}\cos(\cdot+\tilde{K}\log{t}+C_0)}_{H\e{4}(\T)}\leq\frac{C}{t}\quad\textit{for all}\quad t\geq 1
\end{equation}

where 
\begin{equation}\label{defk}
K=\sqrt{12c\e{3}}, \tilde{K}= 9c\e{2},
\end{equation}
 $C_0=C_0(h_0)$ and $C$ depends only on $c$. Moreover, if $v(\cdot,0)\in\mathcal{O}$ we have that 
 \begin{equation}\label{distancia}
 dist_{\dhil{4}}(v(\cdot,t),\mathcal{M}_0)\leq Ce\e{-at}\quad\textit{for any}\quad t>0 
 \end{equation}
 with $a=a(c)$.
\end{thm}
\begin{nota}
Not all the solutions of \eqref{equevol} with initial data close to $c$ have the asymptotic behaviour given in \eqref{cotasolucmanifold}. Indeed, there exists a manifold of solutions such that $h(\cdot+ct,t)-c$ decay exponentially as $t\to\infty$. However, such a behaviour would take place for non-generic initial data.
\end{nota}
\begin{proof}
The proof of Theorem \ref{thestab} is just an application of the results in Theorem \ref{thmmani}. The hypotheses of Theorem \ref{thmmani} are satisfied due to Lemma \ref{lema1}, therefore the manifold $\mathcal{M}_0$ exists. The differential equation that describes the dynamic of $v$ on this manifold is \eqref{v0} which reduces to
\begin{equation}\label{s5e1}
\frac{dv_0}{dt}=P_0R(v_0+\psi(v_0)),
\end{equation}
using that $L_0(v_0)=0$ for $v_0\in\mathcal{E}_0$.

Let $v_0=a_{-1}e^{-i\varphi}+a_{1}e^{i\varphi}\in\mathcal{E}_0$ with $a_1=\overline{a_{-1}}$. Using Lemma \ref{lema2} as well as $P_0\Big(\dfi{}\Big(\frac{v_0^2}{2}\Big)\Big)=0$.
Therefore, equation \eqref{s5e1} becomes
\begin{displaymath}
a'_1(t)e^{i\varphi}+a'_{-1}(t)e^{-i\varphi}=\Big(-\frac{1}{24c^3}-\frac{3i}{4c}\Big)a_1 a_{-1}^2 e^{-i\varphi}+\Big(-\frac{1}{24c^3}+\frac{3i}{4c}\Big)a_{-1}a_1^2 e^{i\varphi}+\mathcal{O}(\norm{v_0}^4).
\end{displaymath}
Henceforth, 
\begin{displaymath}
a'_1(t)=-\alpha\abs{a_1}^2 a_{1}+\mathcal{O}(\norm{v_0}^4)
\end{displaymath}
where $\alpha=\Big(\frac{1}{24c^3}-\frac{3i}{4c}\Big)$.
Using polar coordinates $a_1(t)=\rho(t)e^{i\gamma(t)}$  we have
\begin{displaymath}
\left\{ \begin{array}{ll}
\rho'(t)=-Re(\alpha)\rho^3(t)+\mathcal{O}(\rho\e{4})\\
\gamma'(t)\rho(t)=Im(\alpha)\rho^3(t)+\mathcal{O}(\rho\e{4})
\end{array} \right.
\end{displaymath}
whence, standard ODE arguments yield: 
\begin{displaymath}
\begin{cases}
\rho(t)=\frac{K}{\sqrt{t}}(1+\mathcal{O}(\frac{1}{\sqrt{t}}))\\
\gamma(t)=\tilde{K}\log t+C_0+\mathcal{O}(\frac{1}{\sqrt{t}})
\end{cases}
\end{displaymath}
as $t\to\infty$, where $K=\frac{1}{\sqrt{2Re(\alpha)}}=\sqrt{12c\e{3}}$, $\tilde{K}=\frac{Im(\alpha)}{2Re(\alpha)}=9c\e{2}$ and $C_0=C_0(\gamma(0),\rho(0))$.

Then,
\begin{displaymath}
a_{1}(t)=\frac{K}{\sqrt{t}}e^{i(\tilde{K}\log t+C_0)}(1+\mathcal{O}(\frac{1}{\sqrt{t}}))\quad\textit{as}\quad t\to\infty
\end{displaymath}

Using that on the manifold $\mathcal{M}_0$ we have $v=v_0+\psi(v_0)$ as well as \eqref{f} we obtain,

\begin{align*}
v(\varphi,t)=\frac{2K}{\sqrt{t}}\cos(\varphi+\tilde{K}\log{t}+C_0)+\mathcal{O}(\frac{1}{t})\quad\textit{as}\quad t\to\infty.
\end{align*}

Thus, \eqref{cotasolucmanifold} follows.

Finally, the estimate \eqref{distancia} is a consequence of the global existence result in Theorem \ref{thglobal} and Theorem \ref{thmmani}(iii). 
\end{proof}

\begin{nota}
The asymptotic behaviour in \eqref{cotasolucmanifold} can be reformulated in terms of the original non-dimensional variables (cf. \eqref{rescale}) as:

\begin{equation}\label{defvariablesoriginales}
h(\theta,t)=\lambda c+\frac{2K\lambda}{\sqrt{\varepsilon A\lambda t}}\cos(-\theta+(1-\varepsilon cA\lambda)t+\tilde{K}\log{(\varepsilon A\lambda t)}+C_0)+\mathcal{O}(\frac{1}{t})
\end{equation}
where $\lambda^2=\frac{6\eta\e{2}}{Re(\eta\e{2}-1)b}$, $A=\frac{2\eta\e{2}}{\mu(\eta\e{2}-1)}$ and $t\gg\frac{1}{A\varepsilon\lambda}$. 

Moreover, we recall that in the non-dimensional variables introduced in Section \ref{S3}, the interface separating Fluid 1 and Fluid 2 is given by the curve $r=1+\varepsilon h$. Therefore, an elementary geometrical argument shows that the interface associated to the solutions with asymptotic \eqref{defvariablesoriginales} behaves asymptotically as the circle given by 
\begin{displaymath}
(x-\delta(t)\cos(\theta_0(t)))\e{2}+(y-\delta(t)\sin(\theta_0(t)))\e{2}=r_0\e{2}
\end{displaymath}
where $\delta(t)=\frac{2K\varepsilon\lambda}{\sqrt{\varepsilon A\lambda t}}$, $\theta_0(t)=(1-\varepsilon cA\lambda)t+\tilde{K}\log{(\varepsilon A\lambda t)}+C_0$ and $r_0=1+\varepsilon\lambda c$. Notice that the center of this circle spirals in towards the origin as $t\to\infty$ (cf. Figure \ref{espiral}). The distance between the interface and this circle is of order $\mathcal{O}(\delta(t)\e{2})$ as $t\gg\frac{1}{A\varepsilon\lambda}$. 

If the thin fluid is near the external cylinder (cf. Remark \ref{fluidoexterno}), we obtain the same asymptotic formula \eqref{defvariablesoriginales} with $A=\frac{\mu\eta(1+\eta\e{2})+\mu(\eta\e{2}-1)}{\eta\e{2}(\eta\e{2}-1)}$ and $\lambda\e{2}=\frac{3\mu A}{b Re}$.
\begin{figure}[htb]\label{espiral}
\includegraphics[width=120mm]{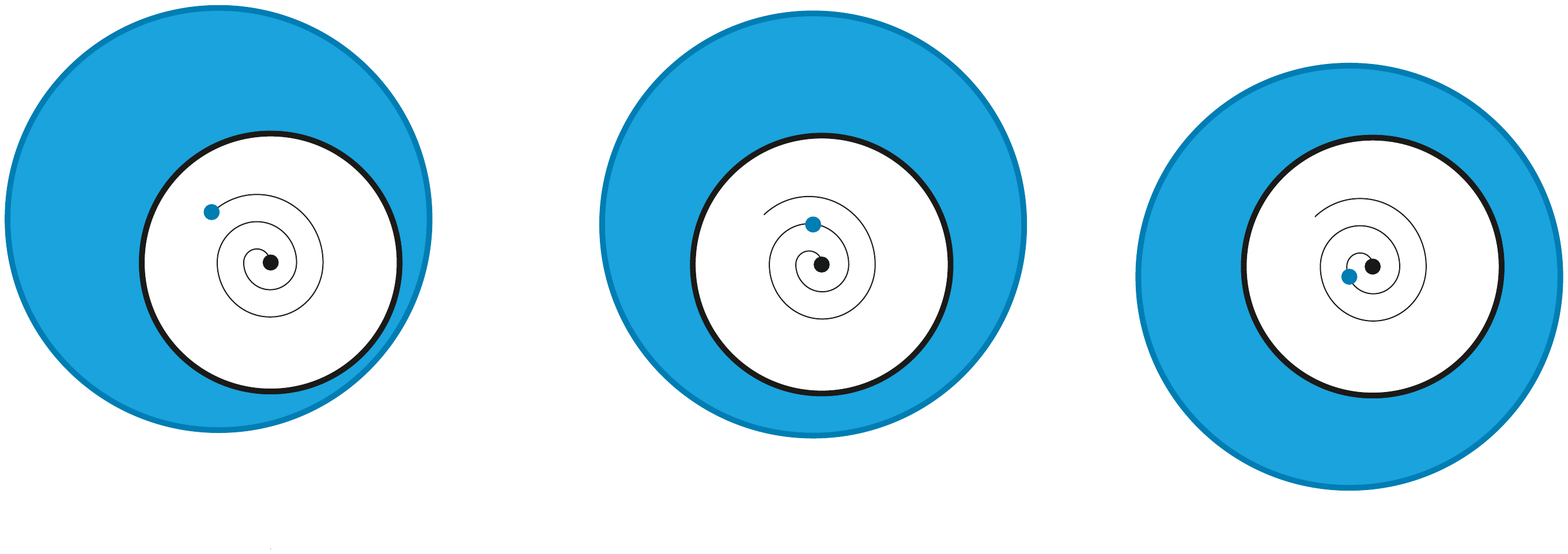}\caption{Center of the interface spiralling towards the center of the cylinders}
\end{figure}
\end{nota}

\subsection{Case $\gamma\gg\frac{1}{\varepsilon\e{2}}$}\label{S5ss2}

We now consider the stability of the solutions of  the equation \eqref{eqevtsgrande}. We have proved in Proposition \ref{prop3} that for any $c>0$ the set of positive stationary solutions $h$ of \eqref{eqevtsgrande} such that $\frac{1}{2\pi}\int_{\T}h=c$, has the following form:
\begin{displaymath}
\mathcal{M}_0(c)=\{h\in\dhil{4}:h(\cdot)=c+c_2\sin(\cdot-\theta_0), c, c_2, \theta_0\in\R, \abs{c_2}<c\}
\end{displaymath}

We reformulate \eqref{eqevtsgrande} using the change of variables $h=c+v(\theta,t)$. Then $v$ satisfies:
\begin{equation}\label{sistmanifgrande}
\frac{dv}{dt}=L(v)+\bar{R}(v)
\end{equation}
where
\begin{equation}\label{operadoresgrande}
\begin{cases}
L(v)=-c\e{3}\Big(\de{2}v+\de{4}v\Big),\\
\bar{R}(v)=-\de{}\Big((v\e{3}+3c\e{2}v+3cv\e{2})(\de{}v+\de{3}v)\Big)
\end{cases}
\end{equation}

\begin{thm}\label{thglobalgrande}
Let $c>0$. There exists $\varepsilon>0$ (depending on $c$) such that, for any $h_0\in H\e{4}(\T)$ satisfying $\norm{h_0-c}_{H\e{4}(\T)}<\varepsilon$ with $\frac{1}{2\pi}\int_{\T}h_0=c$, there exists a unique solution $h\in C([0,\infty); H\e{4}(\T))\cap C\e{1}((0,\infty);H\e{4}(\T))$ of \eqref{eqevtsgrande}, where $h(\cdot,0)=h_0(\cdot)$.

Moreover, we have
\begin{equation}\label{cotathglobalgrande}
dist_{H\e{4}(\T)}(h(\cdot,t),\mathcal{M}_0(c))\leq  Ce\e{-at}\quad\textit{for all}\quad t\geq 0
\end{equation}
where $a$, $C$ are positive constants depend only on $c$.
\end{thm}
\begin{proof}
Local well posedness of \eqref{sistmanifgrande} in the space $C([0,T); H\e{4}(\T))\cap C\e{1}((0,T);H\e{4}(\T))$ will be seen in Appendix \ref{appendix}. 

We decompose $v=v_0+v_1$ with $v_0\in\mathcal{E}_0$ and $v_1\in\mathcal{E}_1$ (cf. \eqref{e0e1}). 
Applying the operator $P_1$ to \eqref{sistmanifgrande} and using that $L$ and $P_1$ commute, we obtain
\begin{displaymath}
\frac{dv_1}{dt}-L(v_1)=P_1\bar{R}(v)
\end{displaymath}

Using the smoothing effect of the equation \eqref{sistmanifgrande} as in the proof of Lemma \ref{regulF1} we obtain that $v\in C\e{\infty}((0,T];\dhil{\ell})$ for $\ell\geq 1$. We can then compute $\frac{d}{dt}\norm{v_1}_{\dhil{4}}\e{2}$ as
\begin{align*}
&\frac{d}{dt}\norm{v_1}\e{2}_{\dhil{4}}+c\e{3}\int_{\T}(\de{6}v_1+\de{8}v_1)\de{4}v_1 d\theta=-\int_{\T}\de{4}P_1\bar{R}(v)\de{4}v_1 d\theta
\end{align*}
Using integration by parts,
\begin{align*}
&-c\e{3}\int_{\T}(\de{6}v_1+\de{8}v_1)\de{4}v_1 d\theta=-c\e{3}\int_{\T}(\de{6}v_1)\e{2}+(\de{5}v_1)\e{2} d\theta\geq -C_0\norm{v_1}\e{2}_{\dhil{6}},\quad\textit{with}\quad C_0>0,\\
&-\int_{\T}\de{4}P_1\bar{R}(v)\de{4}v_1 d\theta=-\int_{\T}P_1(\de{3}((v\e{3}+3c\e{2}v+3cv\e{2})(\de{}v_1+\de{3}v_1)))\de{6}v_1 d\theta\leq C_1\norm{v}_{\dhil{4}}\norm{v_1}\e{2}_{\dhil{6}}.
\end{align*}
Therefore, as long as $\norm{v}_{\dhil{4}}\leq\varepsilon$ with $\varepsilon$ small enough, we arrive at
\begin{displaymath}
\frac{d}{dt}\norm{v_1}\e{2}_{\dhil{4}}\leq -\frac{C_0}{2}\norm{v_1}_{\dhil{4}}\e{2}
\end{displaymath}
as long as $\sup_{s\in[0,t]}\norm{v(\cdot,s)}_{\dhil{4}}\leq\varepsilon$. Thus,
\begin{equation}\label{cotav1grande}
\norm{v_1}_{\dhil{4}}\leq C\norm{v_{in}}_{\dhil{4}}e\e{-\frac{C_0}{4}t}
\end{equation}

Now, applying $P_0$ to \eqref{sistmanifgrande} we obtain
\begin{equation}\label{v0grande}
\frac{dv_0}{dt}=P_0(\bar{R}(v))
\end{equation}
Using \eqref{operadoresgrande} and \eqref{cotav1grande}, we deduce that, as long as $\norm{v}_{\dhil{4}}\leq\varepsilon$, we have that
\begin{align}\label{normap0rgrande}
\norm{P_0(\bar{R}(v))}_{L\e{2}(\T)}\leq C\norm{v}_{\dhil{4}}\norm{v_1}_{\dhil{4}}\leq C\varepsilon\e{2}e\e{-\frac{C_0}{4}t}
\end{align}
whence $\norm{v_0(\cdot,t)}_{\dhil{4}}\leq C(\norm{v_{in}}_{\dhil{4}}+\varepsilon\e{2})$ if $\sup_{s\in[0,t]}\norm{v(\cdot,s)}_{\dhil{4}}\leq\varepsilon$.

Combining this estimate with \eqref{cotav1grande} we readily obtain that $v$ is globally defined in time and $\norm{v(\cdot,t)}_{\dhil{4}}\leq\varepsilon$ for any $t\geq 0$.

Using then \eqref{v0grande} and \eqref{normap0rgrande} it follows that there exists the limit $\lim_{t\to\infty}v_0(\cdot,t)=v_\infty$ in $\dhil{4}$ with $v_\infty\in\mathcal{E}_0$. Moreover, we have $\norm{v_0-v_\infty}_{\dhil{4}}\leq C\varepsilon\e{2}e\e{-\frac{C_0}{4}t}$. Since $\int_{\T}v_\infty=0$  we deduce, using the definition of $\mathcal{E}_0$, that $c+v_\infty\in\mathcal{M}_0(c)$. Thus, using also \eqref{cotav1grande} we obtain $\norm{v-v_\infty}_{\dhil{4}}\leq C\varepsilon e\e{-\frac{C_0}{4}t}$, whence the result follows.

\end{proof}

\begin{nota}
Notice that the solutions in $\mathcal{M}_0(c)$ can be interpreted geometrically for small $\varepsilon$ as circular interfaces with a center shifted slightly from the origin.
\end{nota}

\begin{nota}As we have seen in section \ref{S4}, the equation \eqref{eqevtsgrande} which arises when $\gamma\gg\frac{1}{\varepsilon\e{2}}$ has many more steady states than \eqref{eqevtsgrande1} that appears for $\gamma\approx\frac{1}{\varepsilon\e{2}}$. The only difference between \eqref{eqevtsgrande1} and \eqref{eqevtsgrande} is the presence in the former of the term $\frac{1}{\gamma\varepsilon\e{2}}\de{}{(\frac{h\e{2}}{2})}$ which a lower order term for $\gamma\gg\frac{1}{\varepsilon\e{2}}$. Nevertheless, this term might be expected to drive the circular steady states of \eqref{eqevtsgrande} centred outside the origin to the circular steady states of \eqref{eqevtsgrande} with center at the origin in time scales with $t\gg 1$. However, we will not pursue the analysis of this dynamics in this paper.
\end{nota}
\appendix
\section{Existence of solutions to \eqref{evoleq}-\eqref{defr} and \eqref{sistmanifgrande}-\eqref{operadoresgrande}}\label{appendix}

In this chapter we prove the existence of solutions to \eqref{evoleq}-\eqref{defr} and \eqref{sistmanifgrande}-\eqref{operadoresgrande} locally in time in $\dot{H}^4(\T)$ using energy estimates.

\begin{prop}\label{wellposedness}
(i) For any $c>0$ there exist $\delta_0>0$, $T\in(0,1)$ such that for any $v_{in}\in\dhil{4}$ satisfying $\norm{v_{in}}_{\dhil{4}}\leq\delta_0$, there exist $v\in C([0,T];\dhil{4})\cap C\e{1}((0,T];\dhil{4})$ which solves \eqref{evoleq}-\eqref{defr} with the initial data $v(\cdot,0)=v_{in}$. Moreover, there exists $C$ depending on $c$ such that 
\begin{equation}\label{cotadatoinicial}
\sup_{0\leq t\leq T}\norm{v(\cdot,t)}_{\dhil{4}}\leq C\norm{v_{in}}_{\dhil{4}}.
\end{equation}
The solution $v(\cdot,t)$ can be extended in time as long as $\sup_{0\leq s\leq t}\norm{v}_{L\e{\infty}(\T)}< c$ (equivalently, $\inf_{0\leq s\leq t}(\inf_{\T}h(\cdot,t))>0$) (cf. \eqref{defvfuncionh}).

(ii) For any $c>0$ there exist $\delta_0>0$, $T\in(0,1)$ such that for any $v_{in}\in\dhil{4}$ satisfying $\norm{v_{in}}_{\dhil{4}}\leq\delta_0$, there exist $v\in C([0,T];\dhil{4})\cap C\e{1}((0,T];\dhil{4})$ which solves \eqref{sistmanifgrande}-\eqref{operadoresgrande} with the initial data $v(\cdot,0)=v_{in}$. Moreover, there exists $C$ depending on $c$ such that 
\begin{equation}\label{cotadatoinicial}
\sup_{0\leq t\leq T}\norm{v(\cdot,t)}_{\dhil{4}}\leq C\norm{v_{in}}_{\dhil{4}}
\end{equation}
\end{prop}

The solution $v(\cdot,t)$ can be extended in time as long as $\sup_{0\leq s\leq t}\norm{v}_{L\e{\infty}(\T)}< c$ (equivalently, $\inf_{0\leq s\leq t}(\inf_{\T}h(\cdot,t))>0$) (cf. \eqref{defvfuncionh}).
\begin{proof}
The proof of Proposition \ref{wellposedness} is standard and we only sketch the main ideas on the arguments. We consider the case (i) of the Proposition.
In order to prove existence of solutions we first obtain uniform estimates for the solutions $v\e{\varepsilon}$ of the following regularized problem:
\begin{equation}\label{sistemaregularizado}
\frac{dv\e{\varepsilon}}{dt}=\J L(\J v\e{\varepsilon})+\J R(\J v\e{\varepsilon})
\end{equation}
where $\J$ is a mollifier operator that it is defined as $\J v=\phi_\varepsilon*v$ where $\varepsilon>0$, $\phi_\varepsilon(\cdot)=\frac{1}{\varepsilon}\phi(\frac{\cdot}{\varepsilon})$ with $\phi\geq 0$, $\phi\in C\e{\infty}(\R)$ and $supp(\phi)\subset [-1,1]$. The function $v\in C(\T)$ is considered as a periodic function in $\R$. The operators $L$ and $R$ are as in \eqref{defl} and \eqref{defr}.

The initial value problem for \eqref{sistemaregularizado} with initial data $v\e{\varepsilon}(0)=\J v_{in}\in\dhil{4}$ can be solved for any $\varepsilon>0$ by means of a standard fixed argument. The corresponding solution is in $C\e{\infty}([0,\infty)\times\T)$. We now derive uniform estimates in suitable Sobolev spaces for the solution of \eqref{sistemaregularizado}. To this end, we differentiate four times with respect to the variable $\varphi$, multiply by $\dfi{4}v\e{\varepsilon}$ and integrate by parts to obtain,
 
\begin{align*}
&\frac{d}{dt}\norm{v\e{\varepsilon}}^2_{\dot{H}^4(\T)}\leq c\e{3}\norm{\J v\e{\varepsilon}}_{\dot{H}\e{5}(\T)}\e{2}-c\e{3}\norm{\J v\e{\varepsilon}}_{\dot{H}\e{6}(\T)}\e{2}+J
\end{align*} 
where $J=\int_{\T}\dfi{4}\J R(\J v\e{\varepsilon})\dfi{4}v\e{\varepsilon}d\varphi$.
 
Using the properties of the mollifiers and the fact that $\norm{\J f}_{L\e{2}(\T)}\leq\norm{f}_{L\e{2}(\T)}$ it can be readily seen that
\begin{displaymath}
\abs{J}\leq C(\norm{v\e{\varepsilon}}\e{3}_{\dhil{4}}+\norm{v\e{\varepsilon}}\e{5}_{\dhil{4}}+(\norm{v\e{\varepsilon}}\e{2}_{\dhil{4}}+\norm{v\e{\varepsilon}}\e{4}_{\dhil{4}})\norm{\J v\e{\varepsilon}}_{\dhil{6}}+(\norm{v\e{\varepsilon}}_{\dhil{4}}+\norm{v\e{\varepsilon}}\e{3}_{\dhil{4}})\norm{\J v\e{\varepsilon}}_{\dhil{6}}\e{2})
\end{displaymath}
where $C$ is independent of $\varepsilon$.

Using Young's inequality it follows by means of standard computations that
\begin{displaymath}
\frac{d}{dt}\norm{v\e{\varepsilon}}\e{2}_{\dot{H}\e{4}(\T)}\leq \frac{d}{dt}\norm{v\e{\varepsilon}}\e{2}_{\dot{H}\e{4}(\T)}+\frac{c\e{3}}{2}\norm{\J v\e{\varepsilon}}\e{2}_{\dot{H}\e{6}(\T)}\leq C(\norm{v\e{\varepsilon}}\e{4}_{\dot{H}\e{4}(\T)}+\norm{v\e{\varepsilon}}\e{8}_{\dot{H}\e{4}(\T)})
\end{displaymath}
with $C$ independent of $\varepsilon$. Then a classical Gronwall like argument yields 

\begin{displaymath}
\norm{v\e{\varepsilon}(\cdot,t)}_{\dhil{4}}\leq C\norm{v_{in}}_{\dhil{4}}
\end{displaymath}
for $0\leq t\leq T$ with $T>0$ small enough and $C$ independent of $\varepsilon$. Using a compactness argument and taking the limit $\varepsilon\to 0$ we obtain the existence of solution $v$ of \eqref{evoleq}-\eqref{defr} as stated in the Proposition. The solutions can be extended for later times using similar arguments as long as $h$ remains positive.

For the proof of uniqueness it is more convenient to use the formulation of the problem in terms of the function $h$ defined in \eqref{defvfuncionh}, namely \eqref{equevol}.

Using the fact that $h\e{3}\dfi{3}h=\dfi{}(h\e{3}\dfi{2}h)-3h\e{2}\dfi{}h\dfi{2}h$ we have 
\begin{equation}\label{unicidadecuacion}
\frac{dh}{dt}=-\dfi{}(\frac{h\e{2}}{2})-\dfi{}(h\e{3}\dfi{}h)-\dfi{2}(h\e{3}\dfi{2}h)+3\dfi{}(h\e{2}\dfi{}h\dfi{2}h)
\end{equation}
Suppose that we have two solutions of \eqref{unicidadecuacion}, $h_1, h_2\in\dhil{4}$. By assumption $h_1\geq c_1>0$ and $h_2\geq c_1>0$. Then:
\begin{displaymath}
\frac{d}{dt}(h_1-h_2)=-\dfi{2}(h_1\e{3}\dfi{2}(h_1-h_2))-\dfi{2}((h_1\e{3}-h_2\e{3})\dfi{}h_2)+l.o.t.
\end{displaymath} 
The lower order terms ($l.o.t.$) contains terms having less that four derivatives. Multiplying by $(h_1-h_2)$ and integrating by parts, we have
\begin{align*}
\frac{1}{2}\frac{d}{dt}\norm{h_1-h_2}\e{2}_{L\e{2}(\T)}\leq -\int_{\T}h_1\e{3}(\dfi{2}(h_1-h_2))\e{2}d\varphi+l.o.t.\leq -c_1\e{3}\norm{h_1-h_2}_{\dhil{2}}\e{2}+l.o.t.
\end{align*}

The terms $l.o.t.$ can be estimated by terms of $\int_{\T}(h_1-h_2)\e{2}d\varphi$ or $\int_{\T}(\dfi{}(h_1-h_2))\e{2}d\varphi$, that using interpolation we can estimate as
\begin{displaymath}
\int_{\T}(\dfi{}(h_1-h_2))\e{2}d\varphi\leq \epsilon\norm{h_1-h_2}_{\dhil{2}}\e{2}+C_\epsilon\norm{h_1-h_2}_{L\e{2}(\T)}\e{2}.
\end{displaymath}

Therefore, 
\begin{displaymath}
\frac{d}{dt}\norm{h_1-h_2}\e{2}_{L\e{2}(\T)}\leq C\norm{h_1-h_2}\e{2}_{L\e{2}(\T)}
\end{displaymath} 
and, since $\norm{h_1-h_2}_{L\e{2}(\T)}=0$ at $t=0$, the uniqueness follows.

The proof of the Proposition in the case (ii) follows with a similar argument observing that $\bar{R}$ in \eqref{operadoresgrande} has one term less that $R$ in \eqref{defr}.
 \end{proof}
 
\section*{Acknowledgement}
The authors acknowledge support through the CRC 1060 (The Mathematics of Emergent
Effects) that is funded through the German Science Foundation (DFG), and the Hausdorff
Center for Mathematics (HCM) at the University of Bonn.

\bibliography{referencias}

\begin{thebibliography}{10}

\bibitem{adams_fournier}
Robert~A. Adams and John J.~F. Fournier.
\newblock {\em Sobolev spaces}, volume 140 of {\em Pure and Applied Mathematics
  (Amsterdam)}.
\newblock Elsevier/Academic Press, Amsterdam, second edition, 2003.

\bibitem{Baumert}
B.~M. Baumert and S.~J. Muller.
\newblock Flow regimes in model viscoelastic fluids in a circular {C}ouette
  system with independently rotating cylinders.
\newblock {\em Phys. Fluids}, 9(3):566--586, 1997.

\bibitem{benney}
D.~J. Benney.
\newblock Long waves on liquid films.
\newblock {\em J. Math. and Phys.}, 45:150--155, 1966.

\bibitem{beretta}
E.~Beretta, M.~Bertsch, and R.~Dal~Passo.
\newblock Nonnegative solutions of a fourth-order nonlinear degenerate
  parabolic equation.
\newblock {\em Arch. Rational Mech. Anal.}, 129(2):175--200, 1995.

\bibitem{bernis_friedman}
F.~Bernis and A.~Friedman.
\newblock Higher order nonlinear degenerate parabolic equations.
\newblock {\em J. Differential Equations}, 83(1):179--206, 1990.

\bibitem{bernis_peletier}
F.~Bernis, L.~A. Peletier, and S.~M. Williams.
\newblock Source type solutions of a fourth order nonlinear degenerate
  parabolic equation.
\newblock {\em Nonlinear Anal.}, 18(3):217--234, 1992.

\bibitem{bertozzi_pugh}
A.~L. Bertozzi and M.~Pugh.
\newblock The lubrication approximation for thin viscous films: regularity and
  long-time behavior of weak solutions.
\newblock {\em Comm. Pure Appl. Math.}, 49(2):85--123, 1996.

\bibitem{Belinchon}
G.~Bruell and Granero-Belinchon.
\newblock On the thin film {M}uskat and the thin film {S}tokes equations.
\newblock {\em arXiv:1802.05509}, 2018.

\bibitem{Chandra}
S.~Chandrasekhar.
\newblock {\em Hydrodynamic and hydromagnetic stability}.
\newblock The International Series of Monographs on Physics. Clarendon Press,
  Oxford, 1961.

\bibitem{Chossat}
Pascal Chossat and G\'{e}rard Iooss.
\newblock {\em The {C}ouette-{T}aylor problem}, volume 102 of {\em Applied
  Mathematical Sciences}.
\newblock Springer-Verlag, New York, 1994.

\bibitem{planosinclinados}
Benjamin~P. Cook, Andrea~L. Bertozzi, and A.~E. Hosoi.
\newblock Shock solutions for particle-laden thin films.
\newblock {\em SIAM J. Appl. Math.}, 68(3):760--783, 2007/08.

\bibitem{Hua}
Hua-Shu Dou, Boo~Cheong Khoo, and Khoon~Seng Yeo.
\newblock Instability of {T}aylor–{C}ouette flow between concentric rotating
  cylinders.
\newblock {\em Int. J. Therm. Sci.}, 47(11):1422 -- 1435, 2008.

\bibitem{Drazin}
P.~G. Drazin and W.~H. Reid.
\newblock {\em Hydrodynamic stability}.
\newblock Cambridge Mathematical Library. Cambridge University Press,
  Cambridge, second edition, 2004.
\newblock With a foreword by John Miles.

\bibitem{eidelman}
S.~D. Eidel'man.
\newblock {\em {P}arabolic systems}.
\newblock Translated from the Russian by Scripta Technica, London.
  North-Holland Publishing Co., Amsterdam-London; Wolters-Noordhoff Publishing,
  Groningen, 1969.

\bibitem{Escher}
J.~Escher, Anca-Voichita Matioc, and Bogdan-Vasile Matioc.
\newblock Thin-film approximations of the two-phase {S}tokes problem.
\newblock {\em Nonlinear Anal.}, 76:1--13, 2013.

\bibitem{ferreira_bernis}
R.~Ferreira and F.~Bernis.
\newblock Source-type solutions to thin-film equations in higher dimensions.
\newblock {\em European J. Appl. Math.}, 8(5):507--524, 1997.

\bibitem{fischer}
J.~Fischer.
\newblock Optimal lower bounds on asymptotic support propagation rates for the
  thin-film equation.
\newblock {\em J. Differential Equations}, 255(10):3127--3149, 2013.

\bibitem{giacomelli_gnann}
L.~Giacomelli, M.~V. Gnann, H.~Kn\"{u}pfer, and F.~Otto.
\newblock Well-posedness for the {N}avier-slip thin-film equation in the case
  of complete wetting.
\newblock {\em J. Differential Equations}, 257(1):15--81, 2014.

\bibitem{giacomelli_otto}
L.~Giacomelli, H.~Kn\"{u}pfer, and F.~Otto.
\newblock Smooth zero-contact-angle solutions to a thin-film equation around
  the steady state.
\newblock {\em J. Differential Equations}, 245(6):1454--1506, 2008.

\bibitem{grun}
G.~Gr\"{u}n.
\newblock Droplet spreading under weak slippage---existence for the {C}auchy
  problem.
\newblock {\em Comm. Partial Differential Equations}, 29(11-12):1697--1744,
  2004.

\bibitem{Gunther}
M.~G\"{u}nther and G.~Prokert.
\newblock A justification for the thin film approximation of {S}tokes flow with
  surface tension.
\newblock {\em J. Differential Equations}, 245(10):2802--2845, 2008.

\bibitem{centermanif}
Mariana Haragus and G\'{e}rard Iooss.
\newblock {\em Local bifurcations, center manifolds, and normal forms in
  infinite-dimensional dynamical systems}.
\newblock Universitext. Springer-Verlag London, Ltd., London; EDP Sciences, Les
  Ulis, 2011.

\bibitem{kato}
Tosio Kato.
\newblock {\em Perturbation theory for linear operators}.
\newblock Classics in Mathematics. Springer-Verlag, Berlin, 1995.
\newblock Reprint of the 1980 edition.

\bibitem{Mielke}
A.~Mielke.
\newblock Reduction of quasilinear elliptic equations in cylindrical domains
  with applications.
\newblock {\em Math. Methods Appl. Sci.}, 10(1):51--66, 1988.

\bibitem{Ockendon}
H.~Ockendon and J.~R. Ockendon.
\newblock {\em Viscous flow}.
\newblock Cambridge Texts in Applied Mathematics. Cambridge University Press,
  Cambridge, 1995.

\bibitem{Oron}
A.~Oron, S.~H. Davis, and S.~G. Bankoff.
\newblock Long-scale evolution of thin liquid films.
\newblock {\em Rev. Mod. Phys.}, 69:931--980, Jul 1997.

\bibitem{otto}
F.~Otto.
\newblock Lubrication approximation with prescribed nonzero contact angle.
\newblock {\em Comm. Partial Differential Equations}, 23(11-12):2077--2164,
  1998.

\bibitem{peng}
Jie Peng and Ke-gin Zhu.
\newblock Linear instability of two-fluid {T}aylor–{C}ouette flow in the
  presence of surfactant.
\newblock {\em J. Fluid Mech.}, 651:357–385, 2010.

\bibitem{reed_simon}
Michael Reed and Barry Simon.
\newblock {\em Methods of modern mathematical physics. {I}. {F}unctional
  analysis}.
\newblock Academic Press, New York-London, 1972.

\bibitem{renardy}
Y.~Renardy and D.~D. Joseph.
\newblock Couette flow of two fluids between concentric cylinders.
\newblock {\em J. Fluid Mech.}, 150:381–394, 1985.

\bibitem{manneville}
C.~Ruyer-Quil and P.~Manneville.
\newblock Modeling film flows down inclined planes.
\newblock {\em Eur. Phys. J. B}, 6(2):277--292, 1998.

\bibitem{Schlichting}
H.~Schlichting and K.~Gersten.
\newblock {\em Boundary-layer theory}.
\newblock Springer-Verlag, Berlin, enlarged edition, 2000.
\newblock With contributions by Egon Krause and Herbert Oertel, Jr., Translated
  from the ninth German edition by Katherine Mayes.

\bibitem{Taylor}
G.~I. Taylor.
\newblock Viii. {S}tability of a viscous liquid contained between two rotating
  cylinders.
\newblock {\em Philos. T. R. Soc. Lond.}, 223:289--343, 1923.

\end{thebibliography}
\bibliographystyle{plain}
\end{document}